\newcommand{\Gint}{A} 
\newcommand{\CP}{{\rm{CP}}}
\newcommand{\Bbb}{\mathbb}
\newcommand{\dist}{{\rm{dist}}}
\newcommand{\reals}{\Bbb R}
\newcommand{\complex}{\Bbb C}
\newcommand{\disk}{\Bbb D}
\newcommand{\naturals}{\Bbb N}

\documentclass[12pt]{amsart}  
\usepackage{graphicx,amssymb,amsthm}

\usepackage{verbatim,amsfonts}

\theoremstyle{plain}                    
\newtheorem{thm}{Theorem}[section] \newtheorem{cor}[thm]{Corollary}
\newtheorem{lemma}[thm]{Lemma} \newtheorem{prop}[thm]{Proposition}

\newcounter{ques} \newtheorem{ques}[thm]{Question}
\numberwithin{equation}{section}

\begin{document}
	
\baselineskip=18pt

\title[Approximation by polynomials with only real critical points] 
  {Approximation by polynomials with only real critical points} 

\author {David L. Bishop}
\address{D.L. Bishop\\
         Economics Department\\
         Yale University \\
         New Haven CT 06520-8268}
\email {david.bishop@yale.edu}

\begin{abstract} 
	We strengthen the Weierstrass approximation theorem by proving 
	that any  real-valued continuous function on an interval 
	$I \subset \reals$ can be uniformly approximated by a 
	real-valued polynomial whose only (possibly complex) critical 
	points are contained in $I$.
	The proof uses a perturbed version of the Chebyshev polynomials 
	and an application of the Brouwer fixed point theorem. 
\end{abstract}

\maketitle

\section{Introduction}

The Weierstrass approximation theorem \cite{Weierstrass1885} 
states that for any real-valued,
continuous function  $f$ on a compact interval $I \subset \reals$, and 
for any $\epsilon>0$, there is a real polynomial $p$ so that 
$\lVert f-p\rVert_I = \sup_{x \in I} |f(x)-p(x)| < \epsilon$. The statement does 
not say much about the behavior of  $p$ off the interval $I$, but 
for some applications of Weierstrass's theorem,   it would be 
advantageous to know the location of all the critical points of $p$, 
e.g., in polynomial dynamics the behavior of the iterates of $p$ depends 
crucially on the orbits of all the complex critical points of $p$.
In \cite{BL-Runge}  it is shown that 
one can restrict all the critical points (real and complex)  
to a thin rectangle  $I \times [-\epsilon, \epsilon]$.
In this paper, we prove that we can actually take  $\epsilon=0$.
The following is our main result. 

\begin{thm}[Critically Constrained Weierstrass Theorem]  \label{weier++}
Suppose $f:I \to \reals$ is a continuous function on a compact interval
$I \subset \reals$. Then  for any $\epsilon>0$, 
there is a real polynomial $p$ so that 
$\lVert f-p\rVert_I  < \epsilon$ and $ \CP(p) := \{z \in \complex : p'(z) =0\} 
\subset I$, i.e., every real or complex critical point of $p$ is inside $I$. 
If $f$ is $A$-Lipschitz, then 
$p$ may be taken to be $CA$-Lipschitz for some $C< \infty$ 
independent of $f$.
\end{thm} 

Using dilation and translation, it is enough to prove Theorem
\ref{weier++}  for the
particular interval $I=[-1,1]$, and this 
is the only case we will consider from this point on.
Recall that $f$ is $A$-Lipschitz on $I$  if $|f(x)-f(y)|\leq A|x-y|$
for all $x,y \in I$.
For a Lipschitz function $f$, the derivative $f'$ exists 
and  satisfies $|f'| \leq A$ almost everywhere, and $ f(x) = f(a) 
+\int_a^x f'(t) dt$ (e.g., see Section 3.5 of \cite{MR1681462}).
By the usual Weierstrass theorem, polynomials are dense 
in  $C_\reals(I)$ (the space of  continuous, real-valued
functions on $I$), and every polynomial 
is Lipschitz when restricted to a compact interval, so 
it suffices to prove Theorem \ref{weier++} when 
$f$ is Lipschitz. 

The assumption  that $f$ is real valued is necessary; a 
similar result does not hold for complex valued functions.
Eremenko and Gabrielov \cite{MR1888795} proved that any 
complex  valued polynomial  with only real critical points 
is essentially real-valued itself. More precisely, they 
proved that any such polynomial $p$
is of the form $p(z)=aq(z)+b$,
where $q(z)$ is a real polynomial and $a,b\in \complex$.
(Their result characterizes rational functions with only 
real critical points, but specializes to polynomials as above.)
Polynomials with only real critical values  have played a role in 
various problems, e.g., 
density of hyperbolicity  in dynamics \cite{MR2335796}, 
rigidity of conjugate polynomials \cite{MR1940166}, 
Smale's conjecture on solving polynomial systems \cite{MR2677885}, 
and Sendov's conjecture on the locations of the critical points of a polynomial
in terms of its roots \cite{MR2297082}.

We also note that Theorem \ref{weier++} is non-linear in nature. 
For example, $x^3$ and $(x-1)^3$ each have a single
critical point,  and these are  both  real. 
However, it is easy to check that the sum $x^3 + (x-1)^3 $ has two complex
critical points.   So the set of real polynomials with all
critical points in $I=[0,1]$ is not  a linear subspace of 
$C_\reals(I)$.
Thus many usual methods, such as duality or reducing to 
approximating a spanning set, do not apply. 
Moreover,  Theorem \ref{weier++} need not be true for 
general compact subsets of $\reals$. See Section 
\ref{fails sec} for some disconnected  sets where it fails.

If $p_n \to f$ uniformly, we might hope that $p_n' \to f'$, 
at least when $f$ is analytic, but this is false, except in very 
special cases.
To see why, we first 
recall that the Laguerre-P{\'o}lya class is the collection of entire 
functions (holomorphic functions on $\complex$)
that are limits, uniformly on compact sets,
of real polynomials with only real zeros. These have 
been characterized as follows \cite{Polya1913}: it is the collection of entire
functions $f$  so that (1) all roots are real, (2) the
nonzero  roots satisfy $\sum_n |z_n|^{-2} < \infty$ and  (3) we
have a Hadamard factorization
\begin{align} \label{Hadamard}
	f(z) = z^m  e^{a+bz+c z^2} \prod_n (1 - \frac z{z_n}) e^{z/z_n},
\end{align}
with $m \in \{0,1,2,\dots\}$, $a,b \in \reals$ and $c \leq 0$.
In particular, functions like $\exp(-z^2)$ and $\sin(z)$ are
in  the Laguerre-P{\'o}lya class,
but $\exp(z^2)$ and $\sinh(z)$ are  not.
This class arises in many contexts, e.g., 
the Riemann hypothesis is equivalent to the  claim that a certain
explicit formula defines a function in the Laguerre-P{\'o}lya class 
\cite{MR1827149}, \cite{MR882550}.

A theorem of  Korevaar and Loewner \cite{MR161984}, extending 
earlier work of P{\'o}lya and Laguerre, says that 
if $\{p_n\}$ are polynomials with only real zeros that 
converge uniformly to $f$ on  an interval $I \subset\reals$,
then $f$ must be the restriction to $I$ 
of  a Laguerre-P{\'o}lya entire function, and that $p_n$ converges 
to $f$ on the whole  complex plane  (uniformly on compact sets).
Clunie and Kuijlaars later proved that 
this also holds if we only assume $p_n$ converges 
in measure to $f$ on  a subset $E\subset 
\reals $ of positive measure (see Corollary 1.3, \cite{MR1294323}). 
Recall that $p_n \to f$ in measure
if for every $\epsilon >0$, $|\{x: |f(x)-p_n(x)| > \epsilon \}|$
tends to zero 
(in this paper, $|E|$ will denote the Lebesgue measure of a
measurable subset of $\reals$.)  For bounded intervals $I \subset \reals$, 
pointwise convergence almost everywhere implies convergence in measure, 
so the same conclusion  holds if $p_n \to f$ pointwise on a set $E$ of positive measure. 
Thus there exist analytic functions $f$ on $I$ so that $f'$ 
cannot  be a limit of polynomials with only real zeros, either uniformly,
in measure, or pointwise on a set of positive measure. However, our 
proof of Theorem \ref{weier++} will also give the following result.

\begin{thm} \label{weak limit} 
There is a $C < \infty$ so that 
every bounded, measurable function $f$  on $I$ is the weak-$\ast$
limit in $L^\infty$  
of a sequence of polynomials $\{p_n\}$ with only real zeros, 
and such that   $\sup_n \lVert p_n\rVert_I \leq C \lVert f\rVert_\infty$.
This fails for $C=1$.
\end{thm} 

Here $\lVert g\rVert_\infty$ denotes the $L^\infty$ norm on $I$, and 
$p_n \to f$ weak-$\ast$ in $L^\infty(I, dx)$ if
\[ \int p_n g dx \to \int  fg dx\]
for every Lebesgue integrable function  $g$ on $I$. 

The polynomials constructed in our proof of Theorem \ref{weak limit} will
diverge pointwise almost everywhere, but this is not 
an artifact of the proof; it is forced in many cases. 
We claim  that if $f$  in Theorem \ref{weak limit} is not in the 
Laguerre-P{\'o}lya class, then    $\{p_n(x)\}$  diverges almost everywhere
on the set where $f$ is  non-zero.
To prove this, suppose $f$ is not in the Laguerre-P{\'o}lya class and 
that $\{p_n\}$ is uniformly bounded, has only real zeros,  and
converges  weak-$\ast$ to $f$. 
The Clunie-Kuijlaars theorem stated above 
implies that if $\{p_n\}$ has pointwise limits on a set
of positive measure, these limits must be zero almost everywhere.
However, if a uniformly bounded sequence  $\{p_n\}$ converges 
pointwise  to $0$  on a set  $E$, then 
the dominated convergence theorem  (e.g., Theorem 2.24  of \cite{MR1681462})
implies that 
$\int_E p_n \to 0$.  By weak-$\ast$ convergence of $p_n$ to $f$ we then have 
\[0= \lim_n \int_E p_n = \lim_n \int  p_n\chi_E
=  \int  f \chi_E  =  \int_E f,\]
where we have taken $g=\chi_E$ in the definition of weak-$\ast$ 
convergence. As usual,  $\chi$ is the characteristic (or indicator) 
function of $E$ ($\chi=1$ on $E$ and $\chi=0$ off $E$).
Since this also holds for every measurable subset of $E$, we deduce 
$f$ is zero almost everywhere on $E$ (otherwise either $\int_{E\cap\{f>0\}} f $
or $\int_{E\cap\{f<0\}} f $ would be non-zero). 
This proves the claim.

In order to prove Theorem \ref{weier++},  
we  want to write $f$ as the uniform limit of polynomials
$p_n$ whose derivatives have the form 
\begin{align} \label{product formula} 
	p_n'(x) = C_n \prod_{k=1}^n  (x-z^n_k),
\end{align}
where  $C_n  \in \reals$ and $\{ z_k^n\}_{k=1}^n \subset [-1,1]$.
These points will be perturbations of the roots  of  $T_n$, 
the degree $n$  Chebyshev polynomial (of 
the first kind). We briefly recall the definition. 

Let $J(z) = \frac 12(z + \frac 1z)$ be the Joukowsky map. 
Note that a point $z=x+iy$ on the unit circle is mapped to $x \in [-1,1]$, 
and $J$ is a 1-1 holomorphic  map of $ \disk^* = \{z:|z|>1\}$ to $U=\complex \setminus [-1,1]$.
Thus it has a  holomorphic inverse $J^{-1}: U \to  \disk^*$. 
Then $T_n=J((J^{-1})^n)$ is a $n$-to-$1$ holomorphic map of $U$ to $U$  that 
is continuous across $ \partial U= [-1,1]$.  By Morera's theorem 
(e.g., Theorem 4.19 of \cite{MR4321146}), 
such a function is entire (holomorphic on
the whole plane).  Since $T_n$ is finite-to-$1$, 
Picard's great theorem (e.g., Theorem 10.14 of \cite{MR4321146})
implies  it is a polynomial, and since $T_n$ is $n$-to-1,
the fundamental theorem of algebra implies it must have
degree $n$. Unwinding the definitions, $T_n$  maps $[-1,1]$ into itself, 
takes the extreme values $\pm 1$  at the points 
$\{x_k^n\} =\{\cos( \pi \frac {k}{n})\}_{k=0}^{n}$ (the vertical 
projections of the $n$th roots of unity), and it has its 
roots at $\{r_n^k\} =\{\cos( \pi \frac {2k-1}{2n})\}_{k=1}^{n}$ (the 
vertical projections of the midpoints between the roots of
unity).  
More background and facts about the Chebyshev polynomials 
will  be given in Section \ref{length sec}.

The basic idea of the proof of Theorem \ref{weier++} is to consider 
polynomials as in (\ref{product formula}) where $z^n_k  = r^n_k +y^n_k$ 
are small perturbations of the Chebyshev roots.
Fix a large positive integer $n$ and consider the Chebyshev polynomial $T_n$. 
For $k=1, \dots, n-1$, let  $I_k^n =[r^n_k, r^n_{k+1}]$ 
denote the interval between the $k$th and $(k+1)$st roots of 
$T_n$.  We call these ``nodal intervals'' and call 
the part of the graph of $T_n$ above $I^n_k$ a ``node'' of $T_n$.
Every node of $T_n$ is 
either positive or negative.  Suppose it is positive. 
If we move the roots at the endpoints of 
$I^n_k$ farther apart (but leave all the other
roots of the Chebyshev polynomial 
unchanged), then the node between them becomes higher, and the 
two adjacent negative nodes each  becomes smaller (less negative). Thus the integral of 
the new  polynomial over the union of these three intervals becomes more positive.
See Figure \ref{plot_2pt_full_square}.
This figure, and many others in this paper, was drawn using the
\verb+MATLAB+ program \verb+Chebfun+ by L.N. Trefethen and his
collaborators. See \cite{Driscoll2014}.

\begin{figure}[t]
\centerline{
\includegraphics[width=.6\textwidth]{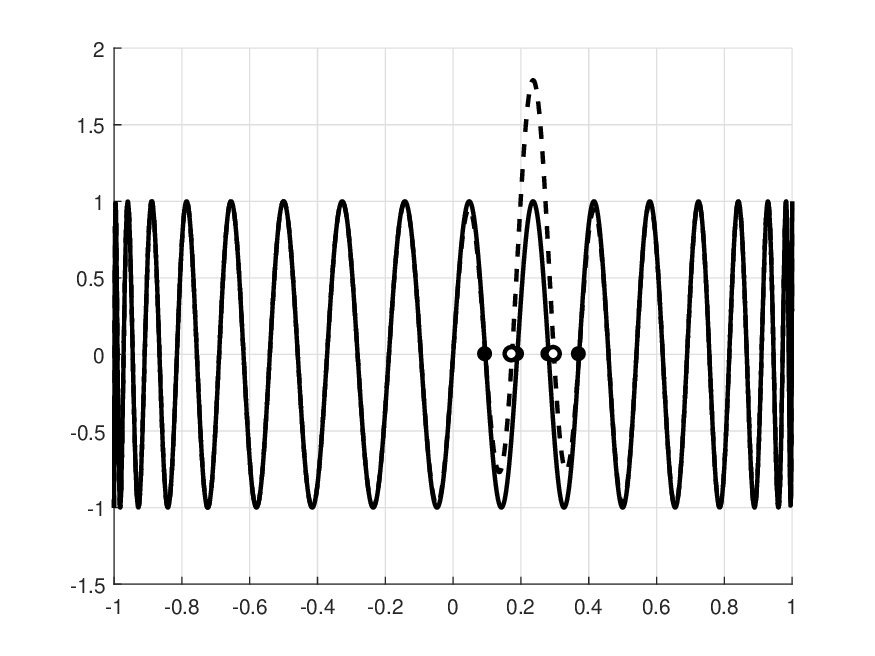}  
\includegraphics[width=.6\textwidth]{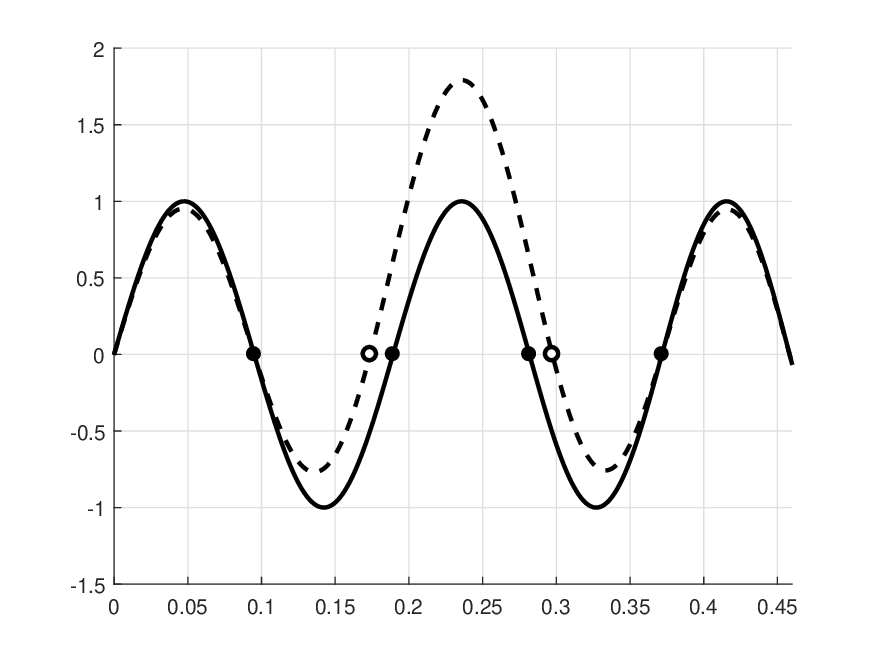} 
} 
\caption{ 
A 2-point perturbation of $T_{33}$. The left picture shows all of 
$[-1,1]$ and the right shows  
an enlargement of the interval where the perturbation occurs. 
The Chebyshev polynomial is solid and the perturbation is dashed.
The white dots are the two new root locations. 
}\label{plot_2pt_full_square}
\end{figure}

When we move  each endpoint of $I_k^n$  by $t|I^n_k|$,
the integral of the polynomial 
over $I^n_k$ changes by at least some  fixed  multiple of  $ t |I^n_k|$. 
A quantitative estimate  like this is one of the key results 
of this paper, although we shall give it for perturbations 
involving three roots instead of only  two. 
We prefer the (more complicated)  3-point perturbations, 
because we can choose them  so 
that the effect on $T_n$ far from the perturbed roots decreases 
more quickly (like $d^{-3}$ instead of $d^{-2}$, where $d$ is the 
distance to the perturbed roots). 
A precise estimate is formulated and proven in Section \ref{3-pt sec}. 
See Figure \ref{Perturb_1bump_sign} for an example of a 3-point 
perturbation.

\begin{figure}[t]
\includegraphics[width=.6\textwidth]{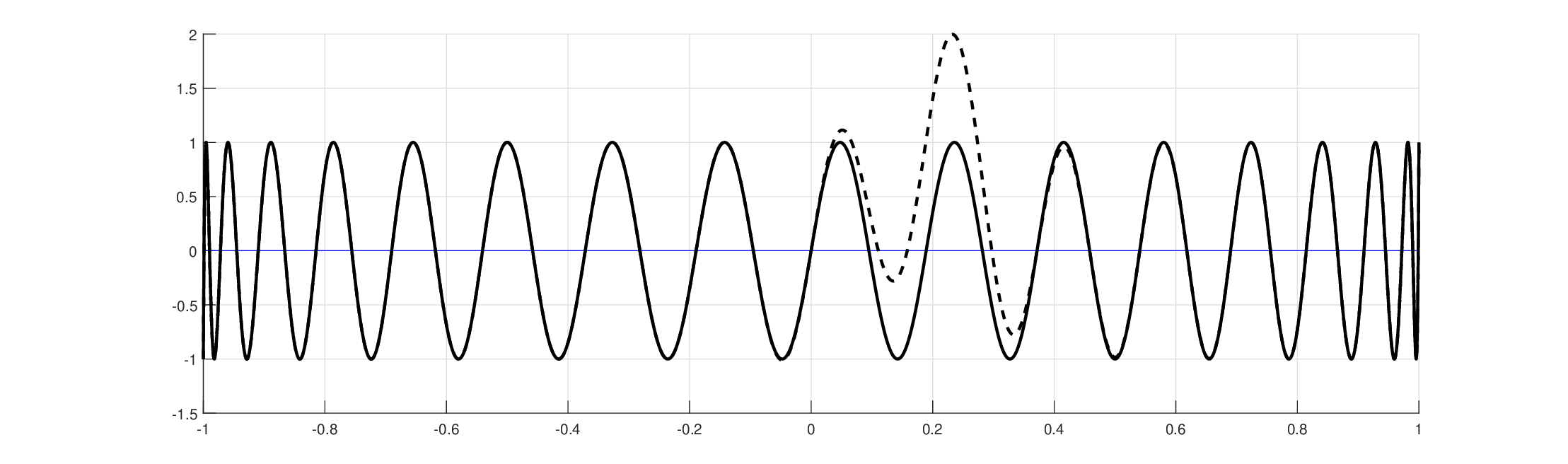}
\includegraphics[width=.6\textwidth]{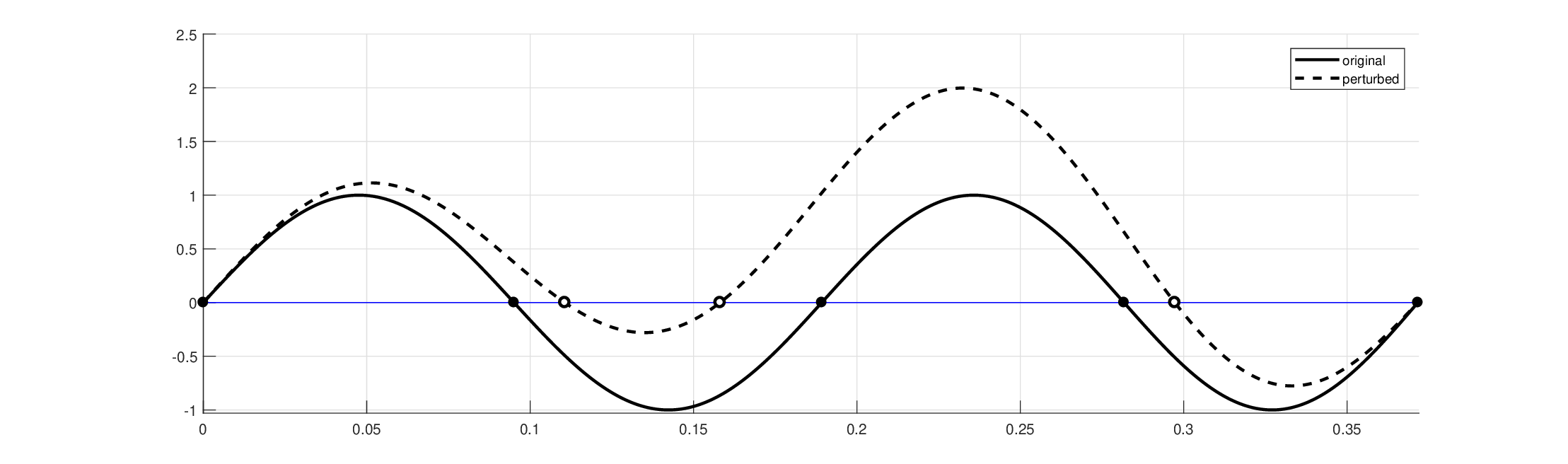}
\caption{ 
A 3-point perturbation. 
The top figure shows $ T_{33}$ 
(solid) and the perturbation $\widetilde T_n$ (dashed) on $[-1,1]$. The bottom figure 
is an enlargement around the perturbed roots. 
}\label{Perturb_1bump_sign}
\end{figure}

See Figure
\ref{Perturb_8bump_sign} for  a degree 33 approximation to 
$f(x) = |x|$.
A degree 201  approximation is shown in Figure \ref{Approx_ABS_deg201}, which 
also shows a log-log plot  showing the rate of approximation  versus
the degree of the polynomial.
Our approximations were chosen by enlarging negative nodes 
to the left of the origin and enlarging positive nodes to the right,
but no attempt was made to do this in an optimal way. 
Nevertheless, the rate of approximation is 
approximately the reciprocal of the degree. This is a little 
surprising, since the best sup-norm approximation of $f(x) =|x|$
by a degree $n$ polynomial (with no restrictions on the critical points) 
satisfies $\lVert f-p_n\rVert \sim (.280169)/n$, e.g.,
see Chapter 25 of \cite{MR3012510}.
Thus our approximations (which are just a first guess)
are fairly close to the best approximation.
Figure \ref{Approx_Lip_1_derv} gives another example
of approximating a  Lipschitz
function by weakly approximating its derivative.

\begin{figure}[htb]
\centerline{ \includegraphics[height=1.6in]{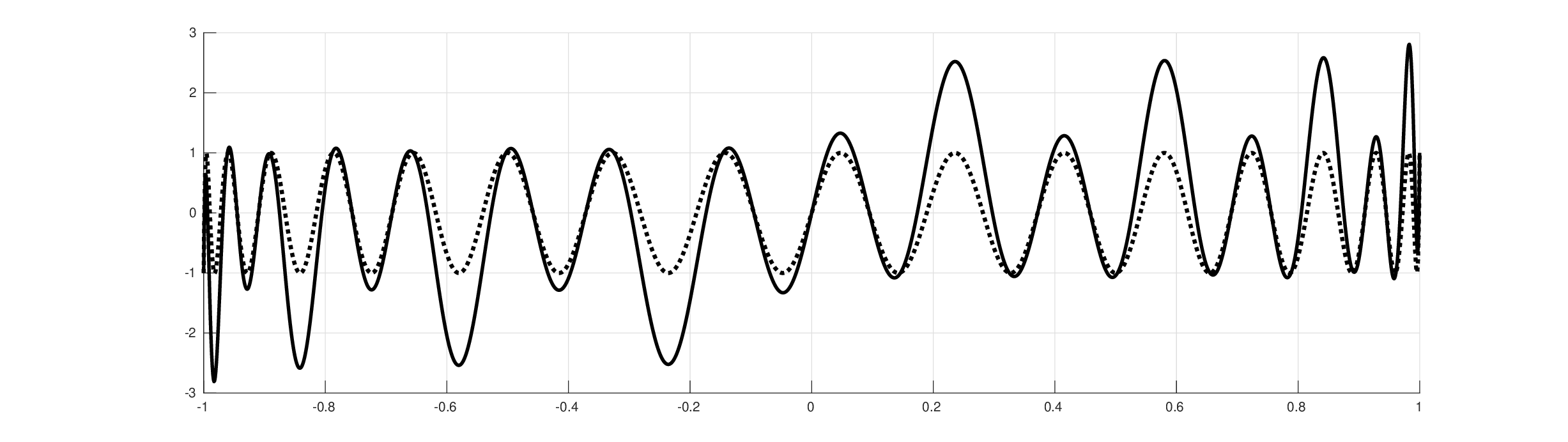} } 
\centerline{ \includegraphics[height=1.6in]{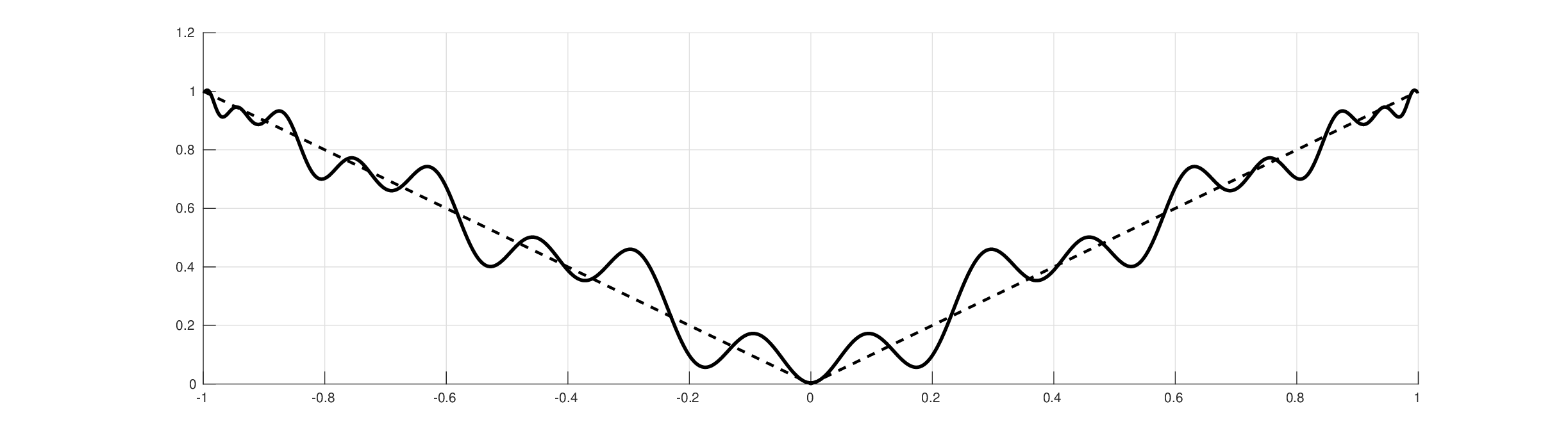} } 
\caption{ \label{Perturb_8bump_sign}
On top  we have perturbed 
	$T_{33}$ (dashed)  to obtain  $p'$ (solid): four pairs 
in $[0,1]$ chosen to make the function more positive, and  four
pairs in $[-1,0]$ chosen to make it more negative. 
The bottom picture shows  $p = \int p'$ (solid), which approximates
$f(x) = |x|$ (dashed).
See Figure \ref{Approx_ABS_deg201} for a higher degree  approximation.
}
\end{figure}

\begin{figure}[htb]
\centerline{ 
\includegraphics[height=2.0in]{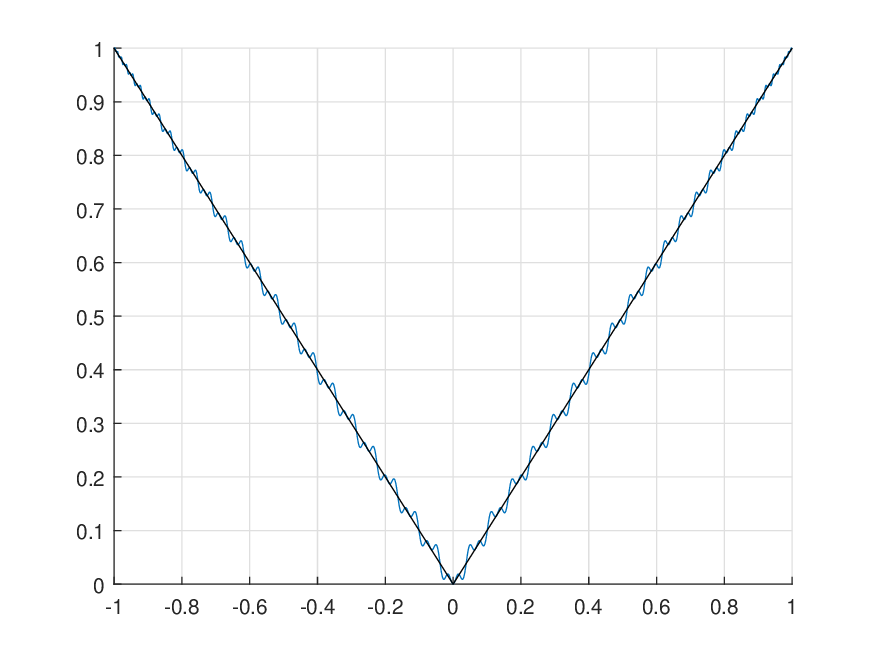}  
\includegraphics[height=2.0in]{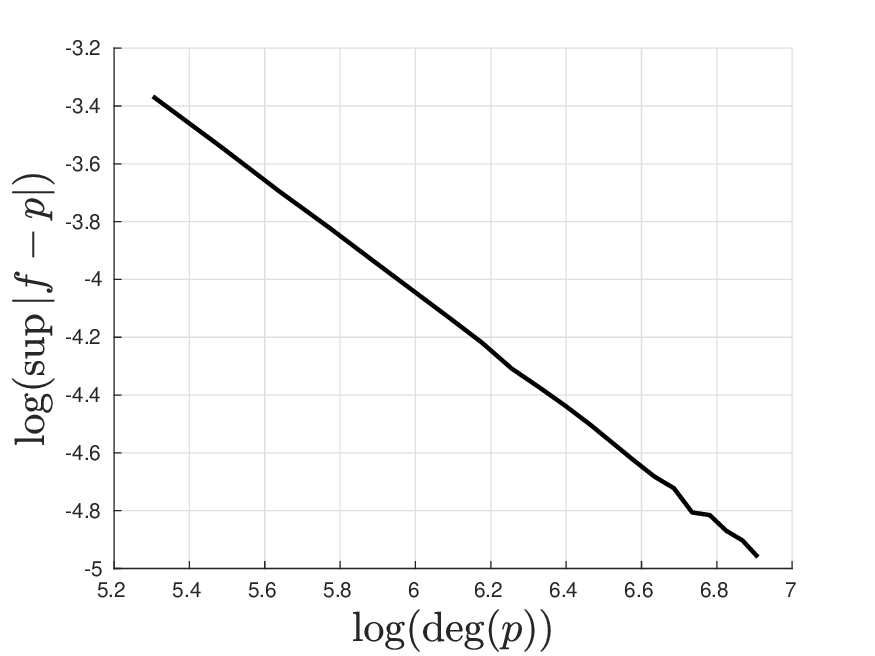}  
}
\caption{ \label{Approx_ABS_deg201}
On the left is a degree 201 polynomial approximating $|x|$.
The right picture is a log-log plot of the sup-norm difference between 
$f(x)=|x|$ and our approximation for degrees between 200 and 1000. 
	The best linear fit is $\approx (-.9912) t+1.8972$. 
The optimal polynomial approximations (no restrictions) behave like 
$\approx -t-1.2724$.
}
\end{figure}

\begin{figure}[htb]
\centerline{
        \includegraphics[height=2.2in]{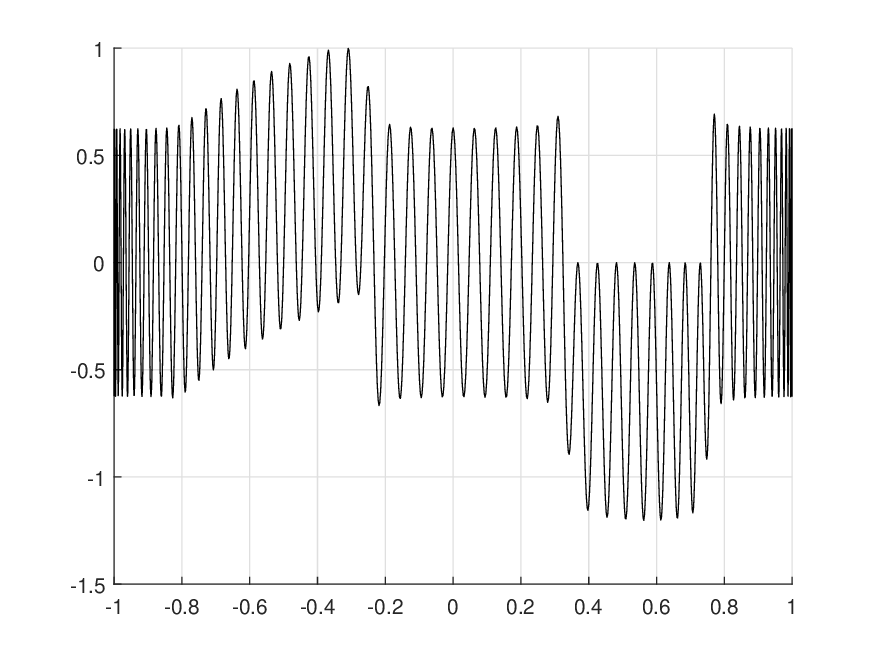}
       \includegraphics[height=2.2in]{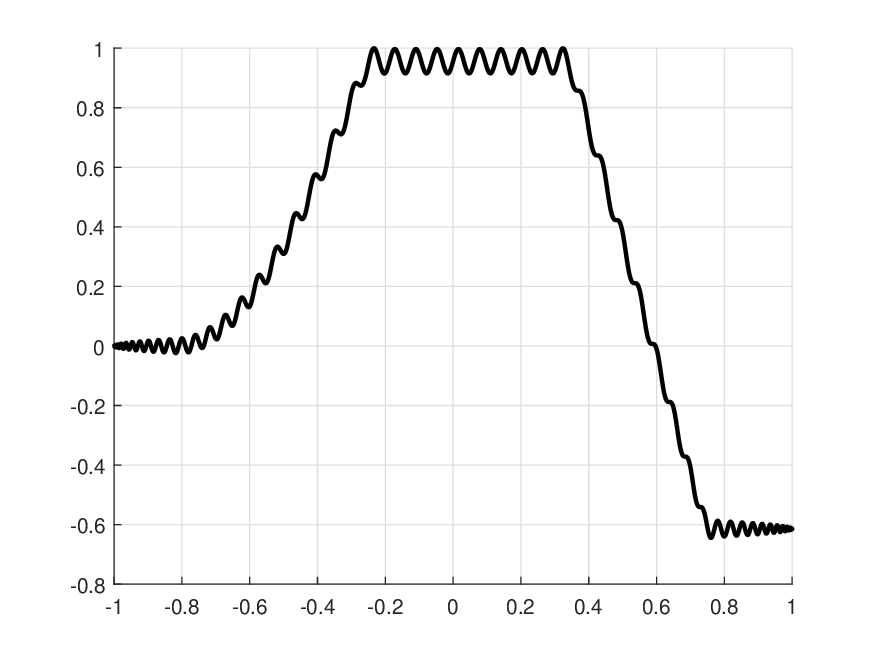} }
\caption{ \label{Approx_Lip_1_derv}
On the left is the perturbed  Chebyshev polynomial
of degree 100, and on the right is its integral. From the picture it seems
clear that any Lipschitz function can be approximated; the goal of the paper
is to prove this is correct.
}
\end{figure}

Briefly, the proof of Theorem \ref{weier++} will proceed as follows.
We  convert  the $n-1$ nodal intervals  into  $N=(n-1)/4$ larger intervals 
$\{G^n_k\}_{k=1}^{N}$,  by taking  unions 
of groups of four adjacent nodal intervals. We would like the origin to be 
the common endpoint of two such intervals, and the  whole arrangement to 
be symmetric with respect to the origin, and this leads us to assume $n-1$ is a
multiple of eight. We then estimate how the Chebyshev polynomial changes when 
we slightly perturb the three interior roots  in a single interval $G^n_k$. 
We make precise the idea that the change is large inside $G^n_k$ and small outside 
this interval (and decays as we move away from $G^n_k$). 
Most of the computations are done when $G^n_k$ is linearly rescaled to be  
approximately $[-2,2]$, 
but these estimates are easily converted to  estimates on the original intervals.
These estimates will show that  there is a $t>0$, so that  for any 
vector $y=(y_1, \dots,y_N)$  with coordinates $|y_k|\leq t$, 
there is a perturbation of the roots 
of $T_n$  that lie in the interior of $G^n_k$
so that  the integral of the perturbed polynomial  
over $G^n_k$ equals $y_k \cdot |G^n_k|$. Of course, perturbations of roots in 
other intervals may destroy this equality, but using the Brouwer fixed point 
theorem, we will show that there is a perturbation of all the roots that gives
the desired equality over every $G^n_k$ simultaneously (except for a bounded
number of exceptions near $\pm 1$). 

In order to prove  Theorem \ref{weier++}, it suffices to consider   functions $f$ 
with a small Lipschitz constant,  e.g.,  less than the
value $t$ chosen above.
For each interval $G^n_k$, we take $y_k =  \Delta(f, G^n_k)/|G^n_k|$ where 
$ \Delta(f,[a,b])   =f(b) -f(a)$. 
Then $|y_k| \leq t$ for all $k$ since $f$ is $t$-Lipschitz. 
Using Brouwer's theorem we can therefore perturb  the roots of $T_n$ to 
obtain a perturbed polynomial $T_n(x,y)$ so 
that  $ \int_{G^n_k} T_n(x,y) dx = y_k |G^n_k|$ for every $k$.
Then any anti-derivative 
$F$ of the perturbed polynomial satisfies $\Delta(F, G^n_k) = \Delta(f, G^n_k)$ 
for every $k$, so choosing an anti-derivative such that $F(0)=f(0)$,
implies that $F$ equals $f$ at every endpoint of every $G^n_k$ (again, with 
a small number of exceptions near $\pm 1$). 
Since both $F$ and $f$  are Lipschitz with 
bounds independent of $n$, and  since $|G^n_k| \to 0$ as $n \nearrow \infty$, 
this implies $F $ uniformly approximates $f$ when $n$ is large enough, 
proving  Theorem \ref{weier++}. 

Roughly speaking, the remainder of the paper divides into four parts.
Part I: Sections \ref{length sec}-\ref{area sec} describe basic properties of Chebyshev 
polynomials and their nodal intervals.
Part II: Sections \ref{perturb sec}-\ref{extreme sec}  
define the perturbations $T_n(x,y)$ of the 
Chebyshev polynomials $T_n$ and give estimates for 
how the perturbed polynomials  compare to $T_n(x)$.
Part III: Sections \ref{interior sec}-\ref{exterior sec} 
verify the conditions needed 
to apply Brouwer's theorem and we prove Theorem \ref{weier++} in Section \ref{fixed pt sec}. 
Part IV gives some auxiliary results: Section \ref{weak sec}  proves
Theorem \ref{weak limit},  Section \ref{diverges ae} shows 
that our polynomial approximants have derivatives that diverge
almost everywhere, and an example of how Theorem \ref{weier++} 
can fail for some disconnected sets is given in Section \ref{fails sec}.

When $A$ and $B$ are both quantities that depend on a common parameter, then 
we use the usual notation $A=O(B)$ to mean that the ratio $B/A$ is bounded independent 
of the parameter. 
The more precise  notation $A=O_C(B)$  will  mean $|A|\leq C|B|$. 
For example $ x = 1 + O_2(\frac 1n)$ 
is simply a more concise way of writing $ 1 -\frac 2n \leq x \leq 1+\frac 2n$.
The notation $A = \Omega_C(B)$ means $A \geq C|B|$ or,
equivalently, $ B = O_C(A)$.
We write $A \simeq B$ if both $A=O(B)$ and $B=O(A)$. 

 I thank  the anonymous referee for a meticulous reading of the manuscript 
and for many very helpful suggestions that improved the clarity and correctness 
of the exposition.

\section{Estimating the length of the nodal intervals}  \label{length sec}

The Chebyshev polynomials defined in the introduction have a number 
of alternate definitions, e.g., see \cite{MR3012510}. For $|x| \leq 1$, we can write 
\[ T_n(x) = \cos( n \arccos(x)),\]
\[ T_n(x) = \frac 12 [(x-\sqrt{x^2-1})^n +(x + \sqrt{x^2-1})^n],\]
\[ T_n(x) = \sum_{k=0}^{\lfloor n/2\rfloor} \binom{n}{2k} (x^2-1)^k x^{n-2k},\]
or $\{T_n\}$ can be defined by the three term recurrence 
\[ T_0(x) =1, \quad  T_1(x) = x, \quad 
T_{n+1}(x) = 2x T_n(x) - T_{n-1}(x) \text{ for } n \geq 1.\]
The latter makes it clear that 
\[ T_n(x) = 2^{n-1}  \prod_k (x -r^n_k),\]
where $\{r^n_k\}$ are the Chebyshev roots  defined in the introduction. 
In particular, the  coefficient of $x^n$ in $T_n$ is $2^{n-1}$. 
Among polynomials of degree $n$ with leading coefficient $2^{n-1}$,
$T_n$ minimizes the supremum norm over $[-1,1]$;  this is the Min-Max 
property (a special case is proven in Lemma \ref{cubic}).
Also note that $\lVert T_n\rVert_\infty = 1$. As mentioned in the
introduction,   $T_n$ has $n-1$  critical points,   all with  singular 
values $-1$ or $1$, and  the endpoints are also extreme points with 
$T_n(1)=1, T_n(-1) = (-1)^n$.
When we perturb the 
roots of $T_n$, some of these extremal values must increase in absolute value, 
and the construction in this paper is based on 
controlling where and how much this happens.

The Chebyshev polynomials are orthogonal with respect to $d \mu = dx/\sqrt{1-x^2}$, 
and expansions in terms of Chebyshev polynomials are  extremely useful 
in numerical analysis; indeed, Chebyshev polynomials are the direct  analog on $[-1,1]$ 
of Fourier expansions on the circle, and many theorems about Fourier expansions 
transfer to Chebyshev expansions. Chebyshev polynomials are also well behaved 
under multiplication and composition, i.e., 
\[ T_n(x) T_m(x) =  \frac 12[T_{m+n}(x) + T_{|m-n|}(x)],\]
\[ T_n(T_m(x)) = T_{mn}(x) .\]
With the additional condition $\deg(T_n)=n$, 
the latter characterizes   Chebyshev polynomials and the 
power functions $\{x^n\}$  \cite{MR1501252} (up to a linear change of variable).

In this paper, we will not need most of the properties above, but we will need
precise estimates of the lengths of the nodal intervals $\{I^n_k\}$, 
and of the integral of $T_n$ over these nodal 
intervals.  The length estimates are addressed in this section and the 
area estimates in the following section. 

Let  $J^n_k = [ \pi \frac {2k-1}{2n}, \pi \frac{2k+1}{2n}]$ for $k=1, \dots, n-1$.
Each of these intervals has length $\pi/n$.
Let $I^n_k = -\cos(J^n_k)$, $k=1, \dots, n-1$, be the  nodal interval between
the $k$th and $(k+1)$st roots of $T_n$, i.e., $I^n_k=[r^n_k, r^n_{k+1}]$. 
(We introduce the minus sign so the $I^n_k$ are labeled left to right in $[-1,1]$.)
Note that when $n$ is even, there are an odd number of nodal intervals and 
$I^n_{n/2}$ contains the origin as its midpoint. When $n$ is odd, the intervals 
$I^n_{(n-1)/2}$ and $I^n_{(n+1)/2}$ share the origin as an endpoint. In both cases, 
the intervals with $k  \leq n/2$ cover $[-1,0]$. 
In our application, we always take $n$ odd, but the estimates in the section
apply to both cases.

Let $|I|$ denote the length of an interval $I$.
Our first goal is to establish some basic facts about  lengths of the 
nodal intervals and the distances between them. 
Note that, by symmetry,  $|I^n_k| = |I^n_{n-k}|$ so most of our estimates 
are only given for $1 \leq k \leq (n-1)/2$, i.e., subintervals of $[-1,0]$.

\begin{lemma}  \label{length lower bound}
	$|I^n_k|\leq  |J^n_k| = \pi/n$.
\end{lemma} 

\begin{proof}
Clearly  $I_k^n$  is  the vertical  projection of $ \exp(i J^n_{n-k})$, 
which has arclength $\pi/n$. 
\end{proof} 

The following says the biggest intervals are adjacent to the origin, 
and that the lengths monotonically 
decrease as we move out towards the endpoints.

\begin{lemma} \label{length monotone} 
	For $x\in [r_1^n, r_{n-1}^n]$, let $I_x$ be the nodal  interval 
	containing $x$ (if $x$ is the common endpoint of two nodal intervals,
	then take $I_x$ to be the nodal interval containing $x$ and closer to $0$). 
	Then $|x|<|y|$ implies $|I_x|\geq |I_y|$. 
\end{lemma} 

\begin{proof}
This is also obvious since the intervals in question 
are vertical projections of 
equal length arcs on the unit circle, and the slope of the circle increases as 
we move toward either $\pm 1$ from $0$. 
\end{proof}

\begin{lemma} \label{length estimate}
For $1 \leq k \leq (n-1)/2$, 
$ \frac {4k}{n^2} \leq |I^n_k| 
	\leq \frac {k \pi^2} {n^2} \approx (9.8696) \frac k{n^2}$.
\end{lemma} 

\begin{proof}
	Note that  $I_k^n$ has left endpoint $ \cos( \pi \frac {2k+1}{2n})$ and 
	right  endpoint $ \cos( \pi \frac {2k-1}{2n})$. 
From the difference rule for cosine, 
\begin{align} \label{diff rule cos}
\cos \alpha - \cos \beta
= -2  \sin  \Big(\frac{\alpha+\beta}2\Big) \sin \Big(\frac{\alpha-\beta}2\Big),
\end{align}
	we can deduce 
\begin{align*}	
	\cos( \pi \frac {2k-1}{2n}) -  \cos( \pi \frac {2k+1}{2n}) 
	= 
	2 \sin( \pi \frac {2k}{2n} )  \cdot 
	\sin(  \frac \pi {2n} ).
\end{align*}	
Now use $ \frac 2 \pi x\leq  \sin x \leq x $ on $[0, \frac \pi 2]$ to derive the estimate 
in the lemma.
\end{proof}

\begin{lemma}  \label{length growth}
With notation as above,  if $1 \leq k \leq k+j \leq n/2$ then 
\[ 1 \leq \frac{|I_{k+j}^n|}{|I_k^n|} \leq 1 + \frac \pi 2 \frac jk.\]
\end{lemma} 

\begin{proof}
	The left hand inequality is just  Lemma  \ref{length monotone}.
Recall  $ J^n_k = [a,b] =   [\pi \frac {2k-1}{2n} , \pi \frac {2k+1}{2n}]$
and  $ J^n_{k+j} = [c,d] =   [\pi \frac {2k+2j-1}{2n} , \pi \frac {2k+2j+1}{2n}]$.
	Using (\ref{diff rule cos}),  we get
\begin{align*}
	\frac{|I_{k+j}|}{|I_k|}
	= \frac {\cos c-\cos d}{\cos a - \cos b} 
	= \frac { \sin\frac{c+d}2  \sin \frac {c-d}2 }
	  { \sin\frac{a+b}2  \sin \frac {a-b}2 }
	= \frac { \sin\frac{c+d}2   }
	 { \sin\frac{a+b}2  },
\end{align*}
	since  $(d-c)/2 = (b-a)/2$. 
	Thus since  $k/n \leq 1/2$, $(\sin x)' = \cos x \leq 1$
	and $\sin x \geq 2x/\pi$ on $[0, \pi/2]$,  we have
\begin{align*}
	\frac{|I_{k+j}|}{|I_k|}
	&\leq \frac { \sin  (\pi(k+j)/n)  } { \sin  \pi k/n  } 
	= \frac { \sin ( \pi k/n)  + \pi j /n  } { \sin  \pi k/n  } \\
	&= 1 +  \frac {  \pi j /n  } { \sin  \pi k/n  } 
	 \leq  1 +  \frac {  \pi j /n  } {  (2/\pi)   \pi k/n  } 
	= 1 +    \frac \pi 2 \frac  j k. \qedhere
\end{align*}
\end{proof} 

From this we can easily deduce that chains of $M$ adjacent   
nodal intervals all
have approximately the same size, at least if $n$ is large,  and if we stay  
away from the endpoints.  More precisely, we have the following result.

\begin{cor} \label{K chains} 
For any $\eta>0$ and $M\in \naturals$, 
there is a $K \in \naturals $ so that 
	if $ K \leq k < k+j < n/2$, then $1 \leq |I_{k+j}^n|/|I_{k}^n| < 1+\eta$
whenever $0\leq  j < M$.
\end{cor} 

If nodal intervals were all exactly the same size, then the distance between 
$I^n_k$ and $I^n_{k+j}$ would be exactly $(j-1) |I^n_k|= (j-1)|I^n_{j+k}|$. 
Because the nodal intervals vary in size,  this is not true, but 
we do have the following similar estimate.

\begin{lemma} \label{dist lower bound}
Suppose $1 \leq k < k+j < n/2$. Then 
	\[ \dist(I^n_k, I^n_{k+j}) \geq  
	  \frac { 2 (j-1)(2k+j) }{ n^2 }    \geq 
	  \frac 4{\pi^2}   (j-1)(1+\frac j{2k})   \cdot |I^n_k| . \] 
\end{lemma} 

\begin{proof}
Let $c = \pi \frac {2k+1}{2n}$ be the right endpoint of $J^n_k$ and
$d=\pi \frac{2k+2j-1}{2n}$ the left endpoint of 
$J^n_{k+j}$. Then  $-\cos(c)$ and $-\cos(d)$ are the right and left
endpoints of $I^n_k$ and $I^n_{k+j}$ respectively, so 
\[ \dist(I^n_k, I^n_{k+j}) = \cos c- \cos d.\]
We can estimate this using the trigonometric identity 
\begin{align*}
	\cos c - \cos d 
	=&  -2 \sin\frac {c+d}{2} \sin \frac{c-d}2  
	 =  2 \sin\frac {c+d}{2} \sin \frac{d-c}2.  
\end{align*} 
Recall that on $[0, \pi/2]$ we have $\sin x \geq \frac 2 \pi x$.
If $0 < c < d \leq  \pi/2$ then $(c+d)/2 \leq \pi/2$ as well, so  
for $j\geq 1$ we get
\begin{align*}
	\cos c - \cos d 
	&\geq   2 \cdot \frac 2 \pi   \frac  {c+d}{2} \cdot \frac 2 \pi   \frac{d-c}2  
	 =    \frac 2{ \pi^2 } (d^2-c^2)   \\
	&=      \frac 12 \cdot (\frac{2k+2j-1}{n})^2-(\frac {2k+1}{n})^2   \\
	&=      \frac 12 \cdot \frac{(2k+2j-1)^2 -(2k+1)^2}{n^2}   \\
	&=      \frac 12 \cdot \frac{(2j-2)(4k+2j)}{n^2}   
	 =    2       \frac{(j-1)(2k+j)}{n^2}  .
\end{align*} 
By Lemma \ref{length estimate},  $|I^n_k| \leq  \pi^2 \frac k{n^2}$, so 
\begin{align} \label{dist lower est}
 \dist(I^n_k, I^n_{k+j})
&\geq  2 \frac{(j-1)(2k+j)}{n^2} \cdot   \frac {|I^n_k|}{ \pi^2 k/n^2}\\ 
&\geq  \frac { 4  (j-1)(2k+j) }{ 2 \pi^2 k}   \cdot |I^n_k|.  \qedhere
\end{align}
\end{proof} 

\begin{cor} \label{ratio upper bound}
	Suppose  $1 \leq k , k+j \leq n$ and  $j\ne 0$. Then 
	\[ \frac{|I_{k}^n|} {\dist(I^n_k, I^n_{k+j})} 
	\leq  \frac{16}{|j|-1}.\]
\end{cor} 

\begin{proof} 
Note that if $I^n_k$ is farther  from the origin than $I^n_{j+k}$ is 
(or is equidistant), 
then all $|j|-1$ nodal intervals between them have length 
at least $|I_k^n|$ and therefore 
\[ \dist(I^n_k, I^n_{j+k}) \geq (|j|-1)|I_k^n|,\]
which is stronger than the inequality in the lemma. 

Otherwise, $I^n_k$ is  strictly closer to the origin than $I^n_{j+k}$. 
First suppose they are on the same side of the origin. 
By symmetry, we may assume they are both to the left of the origin, so  
$I^n_{k+j}$is to the left of $I_k^n$, i.e., $j <0$. In this case,
(\ref{dist lower est}) says 
\begin{align*}
	\frac{|I_{k}^n|} {\dist(I^n_k, I^n_{k+j})} 
&\leq
	\frac {   2 \pi^2 k  |I^n_{k}| } {  2(|j|-1)(2k+|j|)|I^n_{j+k}| }. 
\end{align*} 
	Lemma \ref{length estimate} then gives 
\begin{align*}
	\frac{|I_{k}^n|} {\dist(I^n_k, I^n_{k+j})} 
&\leq 
	\frac {   2 \pi^2 k  (1+ \pi|j|/2k)  } {  2(|j|-1)(2k+|j|)} 
= 
\frac {     \pi^2   (k+ (\pi/2)|j|)  } {  (|j|-1)(2k+|j|)} \\
&\leq 
	\frac {    \pi^2   ((\pi/2)k+ (\pi/2)|j|)  } {  (|j|-1)(k+|j|)}  \\
&= 
\frac {    \pi^3  } {  2 (|j|-1) }
 < 
\frac {    16 } {   |j|-1 }, 
\end{align*} 
as claimed in the lemma.

Finally, we must consider the case when  $I^n_k$ is  strictly 
closer to the origin than $I^n_{j+k}$, but it is  on 
the opposite side of the origin. Then  $I^n_{n-k}$ is 
between $I^n_k$ and $I^n_{k+j}$, has the same size as $I^n_k$, 
and  every  interval between $I^n_k$ and $I^n_{n-k}$ is at least 
this long.  There are $n-2k$ such intervals, including 
$I_{n-k}^n$ but not $I^n_k$ (the picture for $n$ even and 
$n$ odd is slightly different but gives the same number in both cases).
By symmetry we may assume $k \leq n/2$ and $j \geq 2n-k$.  Thus 
\[ \dist(I^n_k, I^n_{k+j}) \geq (n - 2k) |I^n_k| + \dist(I^n_{n-k}, I^n_{j}).\]
Now the previous case applies to the distance between 
$I^n_{n-k}$ and $I_{k+j}^n$, and  since $|I^n_k| = |I^n_{n-k}|$, we get
\begin{align*} 
	\dist(I^n_k, I^n_{k+j})
&\geq(n - 2k) |I^n_k| + \frac 1{16}   (j-(n-2k)-1) |I^n_{n-k}| \\
&\geq    \frac 1{16} [(n - 2k) |I^n_k| + (j-(n-2k) -1)  |I^n_{k}| ] \\
	&=    \frac 1{16}  (j-1) |I^n_k| . \qedhere 
\end{align*}
\end{proof} 

Numerical experiments suggest the corollary holds with the estimate $2/(|j|-1)$, 
although the particular value is not important for the proof of Theorem 
\ref{weier++}; any estimate of the form $O(1/(|j|-1))$ would work.

\section{The area of a Chebyshev node} \label{area sec} 

In this section, we estimate the integral of a Chebyshev polynomial  $T_n$
over a nodal interval $I^n_k$, and show the result is approximately $(2/\pi) |I^n_k|$.

\begin{lemma} \label{project lengths}  
	Suppose $J  =[a,d] \subset [0, \pi]$ has length $t =d-a$ 
	and $J' = [b,c]  \subset J$ is 
	concentric with length $s=c-b$. Then  
	\[ \frac {|\cos(J')|}{|\cos(J)|} \geq \frac st = 
	\frac {|J'|}{|J|}.\]
\end{lemma} 

\begin{proof} 
The intervals being concentric means that 
$(a+d)/2 = (b+c)/2$. The difference formula for cosine implies 
\begin{align} \label{sin/sin}
 \frac {|\cos(J')|}{|\cos(J)|} 
 &= 
 \frac {\cos b  -\cos c}{\cos a - \cos d}   
 = 
 \frac {\sin  \frac{c+b}2  \sin \frac {b-c}2}{\sin \frac{a+d}2  \sin \frac {a-d}2} 
= 
 \frac { \sin \frac {c-b}2}{ \sin \frac {d-a}2}  
 = 
 \frac { \sin s/2}{ \sin t/2 }  ,
\end{align} 
so the claim is equivalent to whether  for $0 \leq s \leq t \leq \pi$ we have 
\[ \frac { \sin s/2}{ \sin t/2 } \geq \frac st.\]
However, this is true because $ \sin(x)/x$ is a decreasing function on $[0, \pi]$, 
 as can be checked by differentiation.
\end{proof} 

\begin{lemma} \label{int lower bound}
	$\int_{I^n_k} |T_n| \geq  \frac 2 \pi |I^n_k|$.  
\end{lemma} 

\begin{proof}
We will use the standard formula 
	(e.g., Proposition 6.24 of \cite{MR1681462})
\begin{align} \label{int formula} 
 \int_I f(x) dx =  \int_0^\infty   |\{ x \in I: f(x) > t\}| dt
\end{align}
that is valid for continuous, non-negative functions.

Recall that for $k=1, \dots, n-1$, we set
$J^n_k = [ \pi \frac {2k-1}{2n}, \pi \frac {2k+1}{2n}]$,
and we defined the nodal intervals $I^n_k = \cos(J^n_k)$.
By the definition of the Chebyshev  polynomials, 
if $x = \cos(y)$, then 
$T_n(x) = \cos(ny)$.
Thus the interval $\{x \in I^n_k:T_n(x) > t\}$
is the image under  cosine of the 
interval $J'=\{y \in J^n_k: \cos(ny)  > t\}$. 
Please also note that it is without loss of generality to assume that $T_n$ is positive on $I^n_k$.

We now apply Lemma \ref{project lengths} with 
$J= J^n_k$, $\cos(J) = I^n_k$, $J' = \{y \in J^n_k: 
|\cos(ny)| > t\}$ and $ \cos(J') = 
\{x \in I^n_k: |T_n(x)| > t\}$. 
The function $|\cos(ny)|$ takes its maximum on $J^n_k$ at 
the midpoint of $J^n_k$, and is symmetric with respect
to this midpoint. Thus $J'$ is concentric with $J$. 
Since $|J^n_k| = \pi/n$,  by Lemma \ref{project lengths} we have 
\[   \frac  {|\{x \in I^n_k:T_n(x) > t\}|} {|I^n_k|} 
	\geq  \frac {|\{y \in J^n_k:\cos(ny) > t\}|}{|J^n_k |} 
	=   \frac n \pi  |\{y \in J:\cos(ny) > t\}| .\]
Using (\ref{int formula}) twice gives
\begin{align*} 
 \frac 1{|I^n_k|}    \int_{I^n_k} T_n(x) dx
   &= \frac 1{|I^n_k|}   \int_0^1  |\{ T_n(x) > t\}| dt \\
	&\geq   \frac n\pi \int_0^1  |\{ \cos(ny)  > t\}| dt \\
	&=   \frac n\pi \int_{J^n_k} \cos(ny)  dy  
   =  \frac n\pi  \frac 2 \pi  \cdot |J^n_k| 
   =   \frac n\pi  \frac 2 \pi  \cdot  \frac \pi n  
	=   \frac 2 \pi . \qedhere
   \end{align*} 
\end{proof} 

A result of Erd{\H o}s and Gr{\" u}nwald 
\cite{MR7}\footnote{This 1939  paper was  one of the first 
listed in {\it Math. Reviews}: it has MR-number 7.} gives a nearby  upper bound. 

\begin{prop}
If $p$ is a polynomial with only real zeros and $a<b$ are roots 
of $p$ with $p>0$ on $(a,b)$ then 
$ \int_I p  dx \leq   \frac 23 |I| \max_I p$.
\end{prop}

Combining this with the previous lemma gives 
\[  .6366  \approx  \frac 2 \pi 
\leq  \frac{\int_{I^k_n} |T_n(x)|dx}{|I^n_k|}  
\leq \frac 23  \approx .6666 .\]
In our situation, we can improve this further.
A direct calculation shows 
the integral actually converges to the lower bound as $n$ increases
(i.e., Chebyshev nodes ``look like'' nodes of sine).

\begin{lemma} \label{area bound} 
With notation as above, 
\[  \frac 2\pi \leq \frac{ \int_{I^n_k} |T_n|}{|I^n_k|} \leq \frac 2 \pi + 
	\frac{ \pi}{  6 n^2} .\]
\end{lemma} 

\begin{proof}
Fix $0< r< 1$, and as above set  $J = J^n_k$ and $J' = \{ y \in  J^n_k : \cos(ny) > r\}$, $0<r<1$. 
Then $\cos(J) = I^n_k$ and  $\cos(J') = \{ x \in I^n_k: |T_n(x)|> r\}$.
Set $s=|J'|$ and $t=|J|$  and  note that $0 < s< t< \pi/n$. 
Using  (\ref{sin/sin}) we have 
\begin{align*} 
\frac{ |\{ x \in I^n_k: |T_n(x)|> r\}| }{|I^n_k|} 
	&=
\frac {|\cos(J')|} {|\cos(J)|} 
= 
\frac {|J'|}{|J|} \cdot  \frac {|J|}{|J'|} \cdot \frac{|\cos (J')|}{|\cos(J)|}  \\
&= 
\frac {|\{y \in J^n_k: |\cos ny|>r\}|} {|J^n_k|} \cdot \frac {t}{s} \cdot \frac{\sin s/2}{\sin t/2} \\
&\leq 
\frac {|\{y \in J^n_k: |\cos ny|>r\}|}{|J^n_k|} \cdot  \sup_{0< s< t< \pi/n}
        \frac{t\sin s/2}{s\sin t/2}.
\end{align*} 
We have already seen that 
	\[ \frac 1{|J^n_k|}  \int_0^1 |\{y \in J^n_k: |\cos ny|>r\}| dr
	= \frac 1{|J^n_k|}  \int_{J^n_k} \cos ny dy = \frac 2 \pi.\]
Using this, (\ref{int formula}),  and $ x-x^3/6  \leq \sin x \leq x$  on $[0, \pi/2]$, 
we can deduce that 
\begin{align*} 
\frac{ \int_{I^n_k} |T_n|}{|I^n_k|} 
&\leq
\frac{ 2}{\pi}  \cdot \sup_{0<s<t<\pi/2n} \frac{t \sin s}{s \sin t}  \\
&\leq \frac 2 \pi \cdot 
	\sup_{0< s <t< \pi/2n} \frac {t s}{ s (t-t^3/6 )}  \\
&\leq \frac 2 \pi \cdot 
	\sup_{0< s <t< \pi/2n} \frac {1 }{ 1-t^2/6 }.
\end{align*} 
This is maximized at $t = \pi/2n$. Thus  the  last line above is less than 
\[
  \frac 2 \pi \cdot \frac{1} { 1-(\pi/2n)^2/6 }  
  =\frac 2 \pi \cdot \frac{1} { 1-\pi^2/24n^2 }  
 \leq  \frac 2 \pi \cdot (1+ \frac {\pi^2 }{12 n^2}   )
= \frac 2 \pi  + \frac {\pi }{6 n^2},   
\]
where the inequality holds  because $ 1/(1-x) \leq 1+2x$ for $0\leq x \leq 1/2$, 
and since $\pi^2/(24 n^2) < 1/2$  for  $n \geq 1$.
\end{proof} 

Because  $T_n$  has opposite signs on adjacent nodal intervals, the integrals 
over adjacent intervals  mostly cancel, especially when the intervals are close 
in length. Lemma \ref{area bound} immediately implies the following estimate 
capturing this.

\begin{cor} \label{int -> 0}
If $I, J$ are adjacent nodal intervals with 
$ |J|\leq |I|\leq (1+\eta) |J|$, then 
\[ \bigg| \frac 1{|I|+|J|}  \int_{I\cup J} T_n(x) dx \bigg| 
\leq 
 \frac 2 \pi \cdot \frac{|I|-|J|}{|I|+|J|}  + \frac \pi{6 n^2}
\leq 
\frac  { \eta} \pi  + \frac \pi{6 n^2} <  \eta/3,   \]
 the last inequality holding for all $n   \geq 6/\sqrt{\eta}$.
\end{cor}

The argument proving Lemma \ref{area bound} applies if $|T_n|$ is replaced by $h(|T_n|)$ for 
some increasing, continuous  function $h$ on $[0,1]$. For example, if $h(x) = x^2$, 
we deduce that 
\begin{align}  \label{T^2 int}
  \frac 12 \leq \frac 1 \pi \int_0^\pi \sin^2(t) dt  \leq \int_{I^n_k}   |T_n|^2(t) dt 
	\leq \frac 12 (1 + \frac {\pi^2}{12n^2}) = \frac 12 + \frac {\pi^2}{24 n^2}  .
\end{align} 
This estimate is utilized in \cite{weier_part2}.

\section{Perturbing the roots }  \label{perturb sec} 

In this section, we introduce notation used in the rest of the paper, 
by defining the subintervals $\{G^n_k\} \subset [-1,1]$ formed
by unions of adjacent nodal intervals, and the polynomials $T_n(x,y)$ 
created by moving the roots of $T_n(x)$ within these intervals.

For convenience, we will assume that $n$ is of the form $n=8m+1$ for 
some positive integer $m$. We do this because 
we are going to group the nodal intervals into groups of four to 
form intervals $G^n_k$, and we want  the origin to be the common endpoint of 
two of these intervals. 
Thus the number of nodal intervals (which is $n-1$), must   be a multiple of 
eight,  leading to the condition $n = 8m+1$. 

Fix  such an  $n$  and let 
$T_n(x)$  be the $n$th Chebyshev polynomial.  
Let $N = (n-1)/4$.
Define $G^n_m$ as the union of the four nodal intervals  
\[G_k^n =  I^n_{4k-3} \cup I^n_{4k-2}  \cup I^n_{4k-1} \cup I^n_{4k} 
=[ r^n_{4k-3}, r^n_{4k+1}], \quad k=1, \dots N.\]
Each $G_k^n$ has roots of $T_n$ as endpoints, and it  contains three
other roots, $r^n_{4k-2}$, $r^n_{4k-1}$, and  $r^n_{4k}$, 
in its interior. We will refer to these as the ``interior roots''  
of $G^n_k$, and label them as  ``left'', ``center'' and ``right'' respectively. 
Since $n$ is odd, $T_n$ is positive on 
the intervals $I^n_k$ with $k$ odd and is negative when $k$ is even. 
Thus on each $G_k^n$ the nodes are positive, negative, positive and 
negative moving left to right. 

For each $G_k^n$ we will perturb  the three interior
roots and leave the endpoint roots of every $G^n_k$  fixed.   
We describe the idea roughly here and more precisely in 
Section \ref{3-pt sec}. 
Given a value $ y \in [0, 1/2]$, moving the roots by a factor 
of $y$ will mean that the rightmost of the three interior 
roots is moved by distance $y \cdot |I^n_{4k-1}|$ to 
the right; if $ y<0$ then this root is moved distance 
$|y| \cdot |I^n_{4k-1}|$ to the left. We can describe both cases
at once by saying the root  is moved by the signed distance $ y|I^n_{4k-1}|$.
The middle and left interior roots  are moved by signed distances
$(-y-\delta) |I^n_{4k-1}|$ and 
$\delta |I^n_{4k-1}|$  respectively, where $\delta$ will be defined
precisely in Section \ref{3-pt sec}. 
Note that this causes the center of mass of the three roots to remained 
unchanged. 
In all the cases of interest,
$\delta$ has the same sign as $y$, $\delta$ is comparable to 
$|y|$, and  $\delta$ depends only on $y$ and the ratio 
$a=|I^n_{4k-2}|/ |I^n_{4k-1}|$, which will be close to $1$.
More concisely, $\delta = \delta(y,a )\simeq y$.
The exact value of $\delta$ is chosen so that the perturbation of the 
three roots simultaneously has the effect of  multiplying the polynomial
by a rational function $R$ so that $R(x) = 1+O(|x|^{-3})$  
as we move away from the perturbations. This decay will be 
verified in Section \ref{3-pt sec} 
and used in Sections \ref{interior sec} and \ref{exterior sec} to show that
the effect of distant perturbations is quite small compared to local ones. 

In order to specify small  perturbations of the Chebyshev roots 
in many different intervals at once, we introduce some notation. 
Let $Q^N_t= [-t,t]^n$ and let 
$\lVert(x_1, \dots , x_n)\rVert = \max_{1\leq k \leq n} |x_k|$ 
denote the supremum norm on $\reals^N$. 
For a vector  $y =(y_1, \dots, y_N)  \in Q_t^N$, let 
$\widetilde y_m = (0, \dots, y_m, \dots,0) \in 
Q_t^N$ be the vector  that equals $y_m$ in the $m$th coordinate 
and is zero in the other coordinates (i.e., the projection of 
$y$ onto the $m$th coordinate axis).
For $y \in Q^N_t$ and $x \in [-1,1]$, 
define $  T_n(x,y)$  to be the  polynomial obtained from 
$T_n(x) $ by  perturbing the zeros of $T_n$ in 
$G^n_k$ by a factor of $y_k$. 
When $y$ is the all zeros vector, we make no perturbations, so $T_n(x,0) = T_n(x)$.
We can give a formula for the 
perturbed polynomial, although it is a bit awkward: 
\[  T_n(x,y) = 2^{n-1} \prod_{k=1}^n (x - z^n_k),\]
where 
\[ 
z^n_k = 
\begin{cases}
r^n_k, \text{ if } k= 1 \mod 4  \quad (\text{endpoints})\\
 r^n_k +y_k |I^n_k|_, \text{ if } k= 0 \mod 4  \quad(\text{right interior}) \\
 r^n_k +\delta(y_k,|I^n_k|/|I^n_{k+1}|) |I^n_k|, \text{ if } k= 2 \mod 4 \quad (\text{left interior}) \\
 r^n_k -(y_k+\delta(y_k,|I^n_k|/|I^n_{k+1}|)) |I^n_k|, \text{ if } k= 3 \mod 4  
	                   \quad(\text{center interior}).
\end{cases} 
\] 
The perturbation is easier to describe in words than in this formula: 
each endpoint of every $G^n_k$ is left fixed ($k=1 \mod 4$); 
the rightmost interior root  ($k=0 \mod 4)$ is moved by 
$y_k$ times the length of the third component sub-interval $ I=I^n_{4k-1}$; the leftmost 
root ($k=2 \mod 4)$  is moved in the same direction by an amount $\delta |I|$, where 
$\delta$ depends on $y_k$ and the length ratio of the center two intervals; the center root
($k=3 \mod 4)$ moves in the opposite direction, and so that the center of mass of the three roots remains unchanged.  

Define $A^n_k(y)$ as the average of $T_n(x,y)$ over $G^n_k$, i.e., 
\begin{align}\label{defn I} 
\Gint^n_k(y) = \frac 1{|G^n_k|} \int_{G^n_k}  T_n(x,  y_k) dx ,
\end{align}
for $k=1, \dots , N = (n-1)/4$. 
The proofs of the  desired estimates for $\Gint^n_k$  will require that
the four nodal intervals making 
up $G^n_k$ to have nearly the same length, to within a fixed factor $1+\eta$.
By Corollary \ref{K chains} we know this holds if we  omit a finite number, $K$,  of
 nodal intervals near each of  $-1$ and $1$.
Thus we will study $\Gint^n_k$ only in the range $K <  k \leq N-K$
(depending on $\eta$, but not on $n$) of nodal intervals near each of  $-1$ and $1$.
The number $K$ depends on $\eta$, but not on $n$. 

Recall $y=(y_1, \dots, y_N)$ and $\widetilde y_k =y_k$ in the $k$th coordinate 
and is zero elsewhere. 
We define maps $f,g$  from 
$ Q^N_t$ into $\reals^N$  whose coordinate functions for $K <  k  \leq  N-K$ 
are given by 
\begin{align}\label{defn f} 
f_k(y) = \Gint^n_k(\widetilde y_k) 
\end{align} 
and 
\begin{align}\label{defn g} 
g_k(y) = \Gint^n_k(y). 
\end{align} 
Note that each $f_k$ considers the integral of different polynomials on 
different intervals  $G^n_k$ (the  polynomial using perturbations  only in $G^n_k$),
whereas $g_k$ is defined using the same polynomial  on every interval  $G^n_k$.

In the remaining dimensions ($k \leq K$ and $k > N-K$),
we simply  let $f_k$ and $g_k$ be the identity maps.
These $2K$  dimensions play no role in the proof, and we might just define 
$f$ and $g$ as maps from $Q^{N-2K}_t$ into $\reals^{N-2K}$. However, 
this complicates the notation, and the indices of the coordinates
of $f$ and $g$ would no longer match the indices of the intervals
$G^n_k$, causing further confusion and inconvenience.  Little is
lost if the reader simply thinks of (\ref{defn f}) and (\ref{defn g}) as 
holding for all $1\leq k \leq N$, but only verifies the  arguments 
in the following sections when $k$ is not too small or too large.

Each coordinate function $f_k$ of $f$ is real valued and 
only depends on $y_k$,  the $k$th coordinate of $y$. Thus we can think of these 
maps as sending intervals to intervals.
We will prove that each  coordinate function  $f_k$ is monotone in
$y_k$ and that it maps $I_t =[-t, t]$  to a strictly larger 
interval; see  Corollary  \ref{f is homeo}.
It follows that  $f$ is a homeomorphism of $Q^N_t$
to some cube  $ Q' \supset Q^N_t$. 
In particular, $Q^n_{t/2} \subset f(Q^n_t)$.
We want to show the same is true
for $g$, i.e., $ Q^N_{t/2} \subset g(Q^N_t)$.
This will imply that given a 
Lipschitz function $F$, we can find a perturbed Chebyshev polynomial with 
an anti-derivative $P_n$ that agrees with $F$ at the endpoints of every 
$\{G^n_k\}$. 
Since both $F$ and $P_n$ are Lipschitz, and the lengths of $G^n_k$ tend to 
zero uniformly with $n$, this will imply that $P_n$ converges uniformly to $F$. 
See Theorem \ref{Lip approx}.

The following result is a precise formulation of the idea that 
the integral of  the perturbed  polynomial $T_n(x,y)$ over the interval 
$G^n_k$ is dominated by the perturbation of the roots inside $G^n_k$, and that
the perturbations exterior to $G^n_k$ have a strictly smaller effect. 

\begin{thm} \label{delta estimate}
There is $t>0$ so that the following holds. 
If $n$ and $K$ are  large enough, and the maps $f,g$ 
	are defined as above  for $y \in Q^N_t$, 
then $Q^N_t \subset f_k(Q^N_t)$ and 
$\lVert f_k-g_k\rVert_{Q^N_t}  \leq t/2$ for $1\leq k \leq N$.
\end{thm} 

This will be proven in Sections \ref{3-pt sec} to \ref{exterior sec}.
If $N$ were equal to $1$, then Theorem \ref{delta estimate} and 
the intermediate value theorem would
immediately imply that $g(Q^1_t)$ covers $Q^1_{t/2}$. In Section \ref{fixed pt sec} 
we will use Brouwer's theorem to draw the same conclusion in the higher 
dimensional case.

\section{The distortion of  2-point perturbations}  \label{2-pt sec} 

Before discussing 3-point perturbations, 
we briefly  discuss moving just two roots. It seems worthwhile to 
do this, since the 3-point perturbation can be 
thought of as the composition of two 2-point perturbations, and the discussion 
of the 2-point perturbation explains the main idea in a simpler setting. 

If a polynomial has zeros at $\pm 1$ and we move these by $\epsilon$ to the 
left and right respectively, this is the same as multiplying the polynomial by 
(see Figure \ref{Rat_Mult_2point})
\begin{align} \label{2-pt R est}
R(x)=	\frac {(x-1-\epsilon)(x+1+\epsilon)}{(x-1)(x+1)} 
	 &=  \frac{x^2 - (1+\epsilon)^2}{ x^2-1}  
	  =  1 -   \frac{ 2\epsilon+ \epsilon^2}{ x^2-1}  .
\end{align} 
Then  $R(x) \geq 1+2 \epsilon + \epsilon^2$ on $[-1,1]$,
and  $0 < R(x) < 1$ on $\{x: |x|>1 +\epsilon \}$.  
 
If we have a polynomial $p$ that has roots at $\pm 1$, and we move 
these roots to $\pm(1+\epsilon)$, then the calculation above says that 
the new polynomial $\widetilde p$  is larger  than $p$ 
(in absolute value)  in $[-1,1]$ and is
smaller (in absolute value)   outside $[-1 - \epsilon,1 + \epsilon]$.
We call the  ratio  $R(x)$  the 2-point distortion function, since it describes how a 
polynomial is altered by moving two of its roots.

We want to generalize this  to a polynomial $p$  with roots at $\{a,b\}$,
that we move slightly  to $\{ a - (b-a)\epsilon/2, b+(b-a)\epsilon/2\}$ (each root
gets moved by $\epsilon$, relative to the length of the interval). 
We say the roots have been ``perturbed by a factor of $\epsilon$.''
Moving the roots apart in this way multiplies $p$ by at least 
$1+2\epsilon + \epsilon2$ in the interval $I$,  and decreases it  outside $I$
in  absolute value by a factor 
\begin{align} \label{distort factor}
	1 -   \frac{ 2\epsilon+ \epsilon^2}{ 4(\dist(x,c)/|I|)^2-1} ,
\end{align} 
where $c =(b+a)/2$ is the center of the interval.  An example of 
perturbing two roots of a Chebyshev polynomial  
was illustrated in Figure \ref{plot_2pt_full_square}. 

\begin{figure}[htb]
\centerline{ 
\includegraphics[height=2.5in]{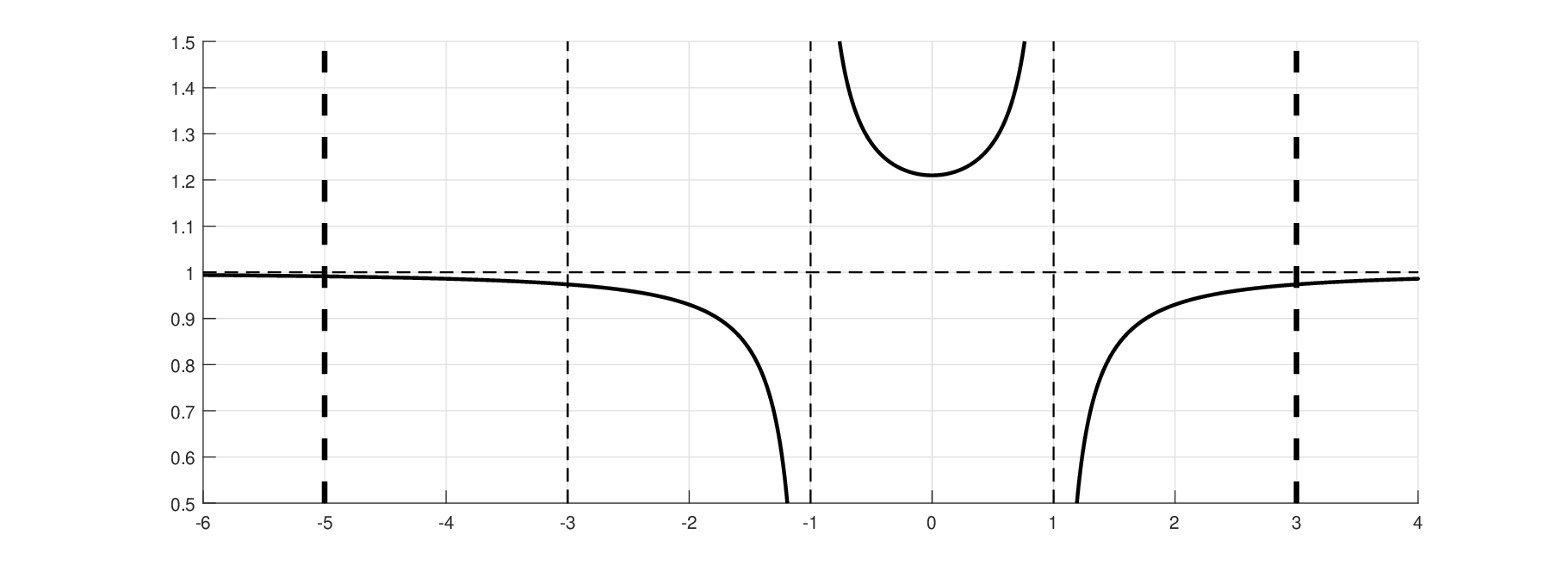}  
}
\caption{ \label{Rat_Mult_2point}
A plot of $r(t) = (a+1+\epsilon)(x-1-\epsilon)/(x^2-1)$. 
If $P$ has roots at $\pm 1$, and we move them by $\pm \epsilon$, 
the new polynomial is $\widetilde  p = R\cdot  p$.  
}
\end{figure}

\section{ The distortion of 3-point perturbations } \label{3-pt sec} 

In this section, we describe the distortion caused by moving three  adjacent roots
of a Chebyshev polynomial.
The two nodal intervals with these endpoints  need not have equal 
lengths,  but  we may assume they  have length ratio close to 1, and 
we will model the situation using three points
$\{ -a, 0,1 \}$ where $a = 1 +\alpha$ with $\alpha$ small, 
say $|\alpha| < 1/10$.
After rescaling, this will cover all the cases that are needed later.

Suppose we have a polynomial $p$  with roots at $\{ -a,0,1\}$,  among 
possibly many other roots.
We create a new polynomial $\widetilde p$ by 
moving these three roots to $-a+\delta$, $-\epsilon -\delta $ and $1+\epsilon$, 
respectively. 
Note that this keeps the center of mass of the roots unchanged  at $(1-a)/3$. 
See Figure \ref{Compare 2 and 3 pts}.

We can think of this  3-point perturbation as the composition 
of two 2-point perturbations:
first moving the pair $\{0,1\}$ to $\{-\epsilon, 1+\epsilon\}$
and then moving the pair $\{-a, -\epsilon\}$  to $\{-a+\delta, -\epsilon-\delta\}$.
Since each move changes $p$ by a factor of $1+O(x^{-2})$  far from the origin, 
the combined motion  also has at least  this decay rate. 
However, by carefully  choosing $\delta$ (depending on $a$ and $\epsilon$),
we can arrange for cancelation that improves the decay rate  to  $O(|x|^{-3})$. 
The remainder of this section explains how to do this.

\begin{figure}[htb]
\centerline{ 
\includegraphics[height=2.0in]{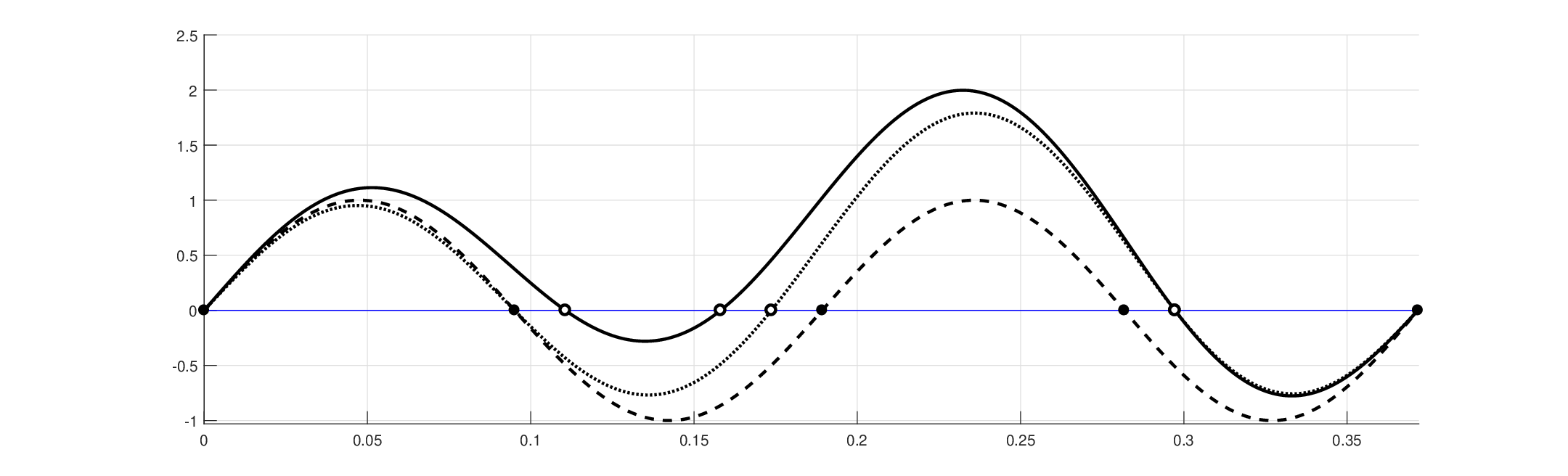}  
	}
\caption{ \label{Compare 2 and 3 pts}
The original Chebyshev polynomial is solid, the 2-point perturbation is dotted
and the 3-point perturbation is dashed.
Each perturbation  is strictly larger than the 
unperturbed polynomial over  the nodal intervals adjacent to the 
perturbed roots, but this need not hold further away: the  2-point 
perturbation is slightly smaller than $T_n$ in  the leftmost  nodal interval.
}
\end{figure}
 
\begin{figure}[htb]
\centerline{
\includegraphics[height=2.4in]{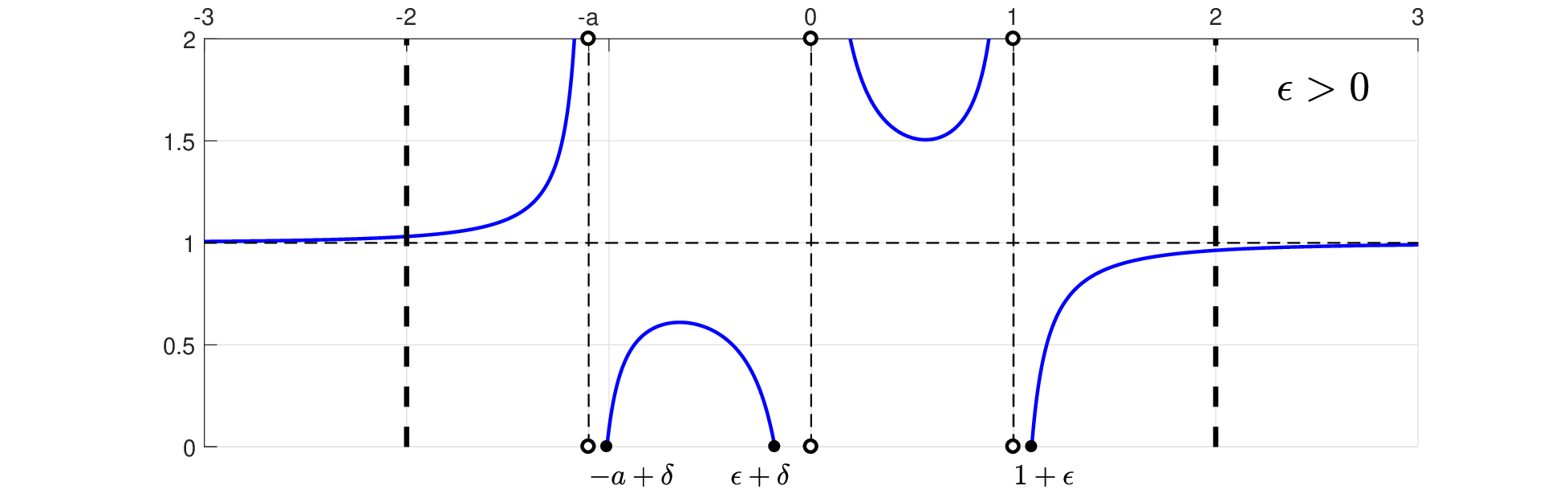}
}
\caption{ \label{Rat_Mult_3point}
The rational function $R$ corresponding to the 3-point perturbation
$\{-a, 0, 1\} \to \{ -a+\delta, -\delta-\epsilon, 1+ \epsilon\}$.
Here we have taken $a= 1.1$ and $\delta= \epsilon = .09>0$.
Over  $[0,1]$, $R$  is bounded strictly above $1$ by an amount 
comparable to the size of the perturbation, and on
the other three nodal intervals, the perturbed polynomial is
larger than the original. 
}
\end{figure}

Note that   $\widetilde p = p \cdot R$ where
\[ 
R(x) = \frac {P(x)}{Q(x)} := 
\frac  {(x+a-\delta)(x+\delta+\epsilon)(x-1-\epsilon)} 
  {(x+a)x(x-1)} .\]
 We call $R$ the 3-point distortion function associated with the perturbation. 
See Figure \ref{Rat_Mult_3point}.
We can write $P(x) = x^3 + Ax^2 + Bx + C$ where 
\begin{align*}
	A &= (a-\delta)+( \delta + \epsilon) +(-1-\epsilon) = a-1,
\end{align*} 
\begin{align*}
	B 
	&= (a-\delta)( \delta + \epsilon) +(a-\delta)(-1-\epsilon)  
	+ (\delta+ \epsilon)(-1-\epsilon)  \\
	&= (a\delta + a\epsilon -\delta^2-\delta\epsilon) + (-a-a \epsilon+\delta + \delta \epsilon)   
	+  (-\delta - \delta \epsilon - \epsilon - \epsilon^2)   \\
	&=  -a - \epsilon + a \delta  - \delta\epsilon  -\delta^2  -\epsilon^2,
\end{align*} 
\begin{align*}
	C 
	&= (a-\delta)( \delta + \epsilon) (-1-\epsilon)  \\
	&= (a \delta+ a \epsilon -\delta^2 - \delta\epsilon)(-1-\epsilon)  \\
	&= -a \delta- a \epsilon +\delta^2 + (1-a) \delta\epsilon
	- a \epsilon^2 +\delta^2 \epsilon + \delta\epsilon^2. 
\end{align*} 
Then
\begin{align*} 
	R(x) &=  \frac{P(x)}{Q(x)} =  1 +  \frac {P(x) -Q(x)} {Q(x)}  \\
	&= 1+\frac {( x^3 + Ax^2 + B x + C) -(x^3 +(a-1)x^2 -ax)}  { (x+a)x(x-1)} \\
	&= 1+\frac { (A-a+1) x^2 + (B+a) x + C) }  { (x+a)x(x-1)} \\
	&= 1+ \frac { b x + C}  { (x+a)x(x-1)} 
\end{align*} 
where 
\[ b=  B+a=  a\delta -\epsilon - \delta\epsilon  -\delta^2  -\epsilon^2.\]
We want to choose $\delta$ so that $b=0$. If $\delta = \lambda \epsilon$, then  
the equation $b=0$  becomes 
\begin{align*} 
	0&=   a  \lambda \epsilon- \epsilon - \lambda\epsilon^2  -\lambda^2 \epsilon^2  -\epsilon^2\\
	0&=  a  \lambda - 1 -\epsilon(\lambda^2 + \lambda +1)\\
	\epsilon &=   \frac{  a  \lambda - 1}{\lambda^2 + \lambda   +1} =  r_a(\lambda).
\end{align*} 
The rational function $r_a$  on the right has
a zero at $1/a$ where it has slope $ \approx 1/3$. More precisely, if
$a=1+\alpha$, then  using long division of polynomials we get 
\[ r_a'(\frac 1a) = \frac{a^3 }{1+a+a^2} 
= \frac 13 \cdot  \frac{1 + 3 \alpha + 3 \alpha^2 + \alpha^3}{ 1 + \alpha + \alpha^2/3}
=\frac 13 [1 +  2\alpha+ \frac  23 \alpha^2 + O(\alpha^3)].\]
So if $\epsilon$ is small and 
$a \approx 1$ (both will hold in our application), then the equation $r_a(\lambda) = \epsilon$
will have a solution of the form 
\[\lambda  = (1/a) + 3\epsilon/a + O(\epsilon^2 + \epsilon \alpha^2).\]
Hence  $\delta = \epsilon/a + O(\epsilon^2)$. 
See Figure \ref{DeltaOfEpsilon}. 
The calculations up to this point prove the following result, giving the 
desired cubic decay of the distortion function $R$.

\begin{figure}[htb]
\centerline{ 
\includegraphics[height=2.5in]{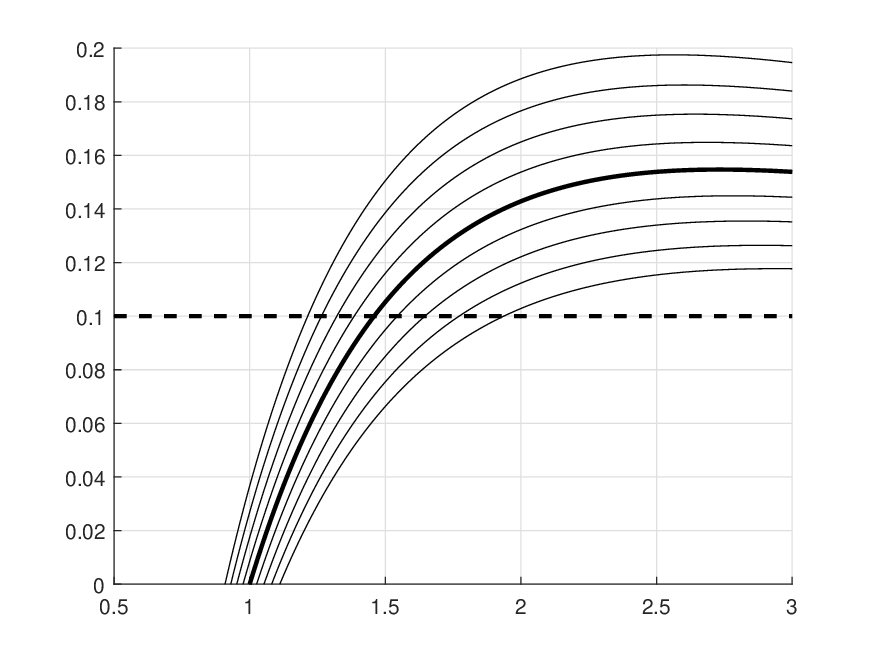}  
}
\caption{ \label{DeltaOfEpsilon}
A plot of $r(t) = (at-1)/(1+t+t^2)$ for $a =\{ .8, .85, \dots 1.2\}$. The curve 
for $a=1$ is thickened. The plots show that  
$r(t) = \epsilon$ has a solution for all $\epsilon \in [0, .1]$ and 
all $a \in [.8, 1.2]$.
}
\end{figure}

\begin{lemma} \label{R = 1+C/Q}
If $|a-1| < 1/5$  and $0\leq \epsilon \leq 1/10$, then we can make a  3-point 
perturbation of the form $\{-a,0, 1\} \to \{ -a+\delta, -\delta-\epsilon, 1+\epsilon\}$ so that 
	$\delta = \epsilon/a + O(\epsilon^2)$,  and the distortion equals 
\begin{align} \label{R expansion}
		R(x) &=  1+ \frac C{Q(x)} = 1 + \frac C{(x+a)x(x-1)}, 
\end{align}
where 
	\begin{align} \label{C estimate} 
C  &=  -a \delta- a \epsilon +\delta^2 + (1-a) \delta\epsilon
- a \epsilon^2 +\delta^2 \epsilon + \delta\epsilon^2 
	= -(1+a) \epsilon + O(\epsilon^2).
\end{align} 	
\end{lemma} 

The following lemma implies the perturbed polynomial  
moves monotonically as a function of the perturbation parameter $\epsilon$.
See Figures \ref{Perturb_3point_plusminus} and \ref{Perturb_3point_plotrange}. 
This will later imply that the function $f$ defined in (\ref{defn f}) is a 
homeomorphism, which will be needed in our application of Brouwer's theorem
(see  Lemma \ref{topology lemma 2}). 

\begin{lemma} \label{monotone} 
Suppose $\widetilde p_1$ and $\widetilde p_2$ are perturbations of 
$p(x) = (x+a)x(x-1)$ as described above, by factors of $\epsilon_1 < \epsilon_2$
respectively. Then $\widetilde p_1 > \widetilde p_2$ for all $x$.
\end{lemma} 

\begin{proof}
Note that 
	\[\widetilde p_1 - \widetilde p_2 = R_1 p - R_2p 
	 =  (1+ \frac {C_1}{Q(x)}) Q(X ) 
	 -(1+ \frac {C_2}{Q(x)}) Q(X )  = C_1 - C_2.\]
Thus $\widetilde p_1 >  \widetilde p_2$ if and only if
 $C_1 > C_2$.
Thus we either have $\widetilde p_1 > \widetilde p_2$ everywhere, 
or  $\widetilde p_1  \leq  \widetilde p_2$ everywhere.
Since  $\epsilon_1< \epsilon_2$, the rightmost  perturbed  root of 
$\widetilde p_1$
is to the left of the  corresponding   root  $y$ of $ \widetilde p_2$, and hence   
$\widetilde p_1(y)> \widetilde p_2(y)=0$. Thus  $\widetilde p_1 > \widetilde p_2$  
everywhere. 
\end{proof} 

\begin{figure}[htb]
\centerline{ 
\includegraphics[height=2.4in]{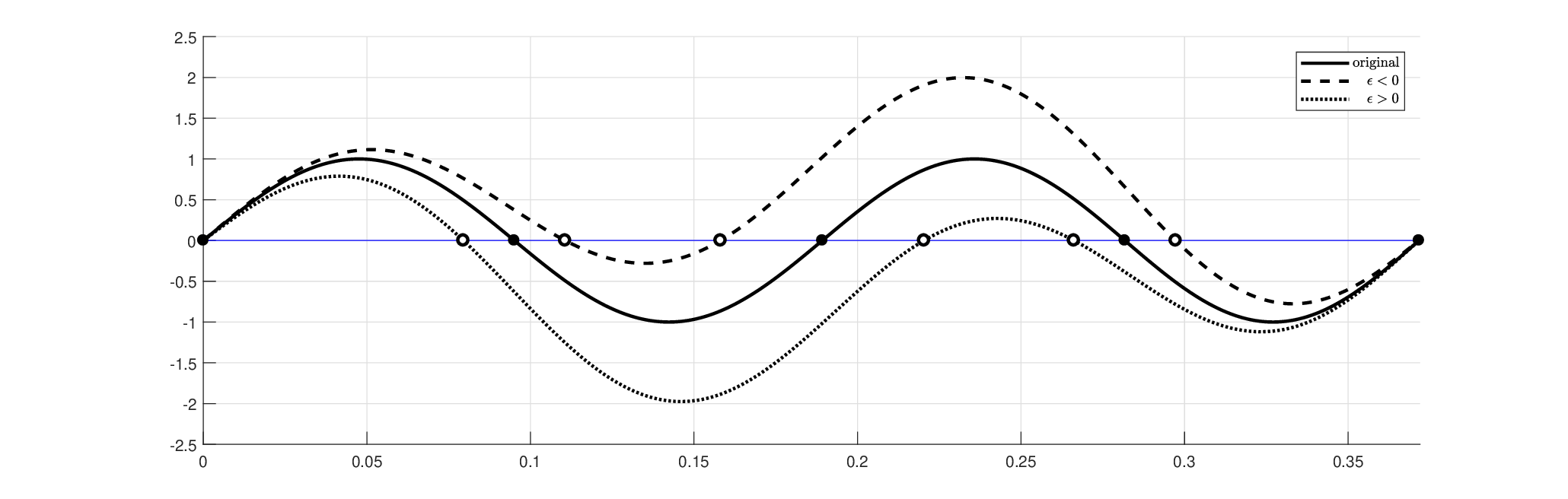}  
}
\caption{ \label{Perturb_3point_plusminus}
A 3-point perturbation for both positive and negative values of $\epsilon$. 
The solid graph is the unperturbed Chebyshev polynomial over one group of 
four nodal intervals. The dashed line is the $\epsilon>0$ perturbation 
and the dotted is the $\epsilon<0$ perturbation.
This illustrates the monotonic movement proven in  Lemma 
\ref{monotone}.
A wider range of perturbations is shown in Figure \ref{Perturb_3point_plotrange}. 
}
\end{figure}

Note that  in (\ref{C estimate}), 
if  $\epsilon$ is  small and $a$ is close to $1$ (which is the case in 
our applications), then  we have $C \approx -2 \epsilon$.
In particular, for $\epsilon$ sufficiently small, $R(x)-1$ has the opposite sign as 
$Q(x) = (x+a)x(x-1)$. Therefore we get the following inequalities.
See Figure \ref{Rat_Mult_3point}.

\begin{cor}  \label{R est}
For $|\epsilon|$ small enough, 
\begin{align*}
	R(x) \geq 1, &\text{ if }   \epsilon >0 \text{  and  } x \in (-\infty , -a] \cup [0,1],\\
	R(x) \geq 1,  &\text{ if }   \epsilon < 0 \text{  and  } x \in (-a,0] \cup [1,\infty), \\
	R(x) \leq 1,  &\text{ if }   \epsilon <0 \text{  and  } x \in (-\infty , -a] \cup [0,1], \\
	R(x) \leq 1,   &\text{ if }   \epsilon > 0 \text{  and  } x \in (-a,0] \cup [1,\infty).
\end{align*} 
\end{cor}

We shall give separate estimates  for $R(x)$ 
on different sets of intervals based on their distance to the origin: 
we call these cases   ``near'' ($[-a,1]$), ``intermediate'' ($[-2,-a]$ and $[1,2]$) and 
``far'' ($\{x: |x|\geq 2\}$). 
Since $1+a \approx 2$ if $a \approx 1$ and $|(x+a)x(x-1)| \approx 6$
at $x= \pm 2$, we can immediately deduce the following 
for the intermediate intervals. 

\begin{cor}  \label{R est 0}
For $|\epsilon|$ small enough and $a = 1+\alpha$ close to $1$, we have 
\begin{align*}
R(x) &\geq 1 + [\frac 13 +O(\alpha) +O(\epsilon)] \epsilon > 1, 
   \text{ if }   \epsilon >0 \text{  and  } x \in [-2,-a] ,\\
R(x) &\leq 1 - [\frac 13 +O(\alpha) +O(\epsilon)] \epsilon  < 1,
  \text{ if }   \epsilon >0 \text{  and  } x \in [1,2] ,\\
R(x) &\leq 1 + [\frac 13 +O(\alpha) +O(\epsilon)] \epsilon  < ,
  \text{ if }   \epsilon <0 \text{  and  } x \in [-2,-a] ,\\
R(x) &\geq 1 - [\frac 13 +O(\alpha) +O(\epsilon)] \epsilon  > ,
  \text{ if }   \epsilon <0 \text{  and  } x \in [1,2]. 
\end{align*}
\end{cor} 

Next  we estimate the distortion $R(x)$  on the ``far'' intervals.

\begin{cor} \label{R est far}
	Suppose $a = 1+\alpha$ with $|\alpha|< 1/5$. Then for $|x| \geq 2$, 
	\[R(x) 
	= 1 -  \frac { [2+O(\alpha)+O(\epsilon)] \epsilon   } 
	 { (x+1)x(x-1)}  .
	 \]
\end{cor} 

\begin{proof} 
	By Lemma \ref{R = 1+C/Q}, 
\begin{align*}
R(x) 
	&= 1+ \frac {  -(1+a) \epsilon + O(\epsilon^2) } 
	 { (x+a)x(x-1)} \\
	&= 1- \frac { (1+a) \epsilon + O(\epsilon^2) } 
	 { (x+1)x(x-1)}  \cdot  \Big(\frac{x+a+1-a}{x+a} \Big)  \\
	&= 1- \frac { (1+a) \epsilon + O(\epsilon^2) } 
	 { (x+1)x(x-1)}  \cdot \Big(1 + \frac{1-a}{x+a} \Big).
\end{align*} 	
	The maximum absolute value of $ 1+ (1-a)/(x+a)$ on 
	$\{x:|x| \geq 2 \}$ is attained at either $x=2$ (if $a>1$)  or $x=-2$ (if $a<1$), 
	where the values are, respectively, 
	\[ 1 + \frac {1-a}{2+a} = \frac 3{2+a} = 1  - \frac  \alpha3  + O(\alpha^2), \]
	\[ 1 + \frac {1-a}{-2+a} = \frac {-1}{2-a} = -1  - \alpha + O(\alpha^2).\]
	In either case we deduce that the maximum  absolute value
	is $1 + O(\alpha)$. Thus 
\begin{align*}
R(x) 
	&= 1- \frac { (2+\alpha) \epsilon [1+O(\alpha)  + O(\epsilon^2)] } 
	 { (x+1)x(x-1)}  \\
	&= 1- \frac { (2+O(\alpha)) \epsilon  + O(\epsilon^2) } 
	 { (x+1)x(x-1)}.  \qedhere   
\end{align*} 	
\end{proof} 

Next we consider the distortion near the perturbed points. 
There are separate estimates for each sub-interval $[0,1]$ and $[-a,0]$. 

\begin{lemma} \label{R est 1}
If  $\epsilon >0$, $a =1+\alpha$   and $x \in [0,1]$,  then
\[ 
R(x) \geq   1 +  ( 4 - \frac 1 {a^2}) \epsilon + O(\epsilon^2)
	\geq 1+ [3 +O(\alpha)+O(\epsilon)]\epsilon  > 1.
\]
\end{lemma} 

\begin{proof}
For $x  \in[0,1]$, $R(x) \geq 1$ and  we can write
\begin{align*}
 R(x)
&= ( 1- \frac {\delta}{ x+a}) ( 1+ \frac {\epsilon+\delta}{x}) (1- \frac{\epsilon}{x-1}) \\
&= ( 1- \frac {\delta}{ x+a}) ( 1+ \frac {\epsilon+\delta}{x}) (1+ \frac{\epsilon}{1-x}),
\end{align*}
and all three terms in the last line are positive for $0\leq x \leq 1$.
Thus we make the second term smaller 
by subtracting the  positive term $\delta/x$, and so
\begin{align*}
 R(x)
&\geq ( 1- \frac {\delta}{ a}) ( 1+ \frac {\epsilon}{x}) (1+ \frac{\epsilon}{1-x}).
\end{align*}
By symmetry, the function on the right  takes a minimum at the midpoint of the two poles, 
	i.e., at $x=1/2$ (one can also verify this by differentiating).
Thus on $[0,1]$, $R$ is bounded below by
\begin{align*}
( 1- \frac {\delta}{ a})
	( 1+ \frac {\epsilon}{1/2}) (1+ \frac{\epsilon}{1-(1/2)})
	&= (1- \frac {\delta}{ a}) ( 1+2\epsilon) (1 +2\epsilon )\\
	&= (1- \epsilon a^{-2}  +O(\epsilon^2 ))( 1+  4\epsilon+4\epsilon^2 ) \\
	&=  1 + ( 4 - \frac 1{a^2})  \epsilon  +O(\epsilon^2) . \qedhere
\end{align*}
\end{proof}

\begin{lemma} \label{R est 2}
If  $\epsilon >0$, $a =1+\alpha$ and $x\in [-a,0]$  then
\[ 
R(x) \leq   1 -  (\frac 4{a^2} -1) \epsilon + O(\epsilon^2) 
	\leq 1- (3 +O(\alpha)+O(\epsilon)) \epsilon < 1.
\]
\end{lemma}

\begin{proof} 
For $x \in [-a,0]$,  $R(x) < 1$ and  we have 
\begin{align*}
 R(x) &=
( 1- \frac {\delta}{ x+a}) ( 1+ \frac {\epsilon+\delta}{x}) (1- \frac{\epsilon}{x-1}) 
 = ( 1- \frac {\delta}{ x+a}) ( 1+ \frac {\epsilon+\delta}{x}) (1+ \frac{\epsilon}{1-x}). 
\end{align*}
Since $\epsilon >0$, the last term in the product is $\geq 0$ for 
$ x \in [-a,0]$, but the other two  terms may 
be positive or negative. 

For example, if the second term is negative, then $x \in (-\epsilon-\delta, 0)$. 
In this case, the first term is approximately equal to $1+\delta/a$,  and since 
$\delta $ is positive if $\epsilon$ is positive, 
this term of the product must be 
negative. Therefore the whole product is negative in this case,
so $R(x) < 0  <  1$ and the estimate in the lemma is certainly true. 

If the first term in the product is negative then $x \in (-a,-a+\delta)$, and 
the second term is approximately $1-(\epsilon+\delta)/a$. This is positive 
since $\epsilon$ and $\delta$ are both small and $a \approx 1$. Again, the 
whole product must therefore be positive in this case, and the lemma holds. Thus
we can restrict attention to the case when all three terms of the product 
are positive, i.e., $-1+\delta < x< -\epsilon - \delta$. 

In this case, 
the middle term of the product is made larger by removing the negative term $\delta/x$, 
and the third term increases by setting $x=0$. 
Thus  on $[-a,0]$ we have 
\begin{align*}
	R(x) &\leq ( 1- \frac {\delta}{ x+a}) ( 1+ \frac {\delta}{x}) ( 1 + \epsilon).
\end{align*} 
By symmetry, this function of $x$ takes its maximum at the midpoint of the two poles, 
i.e. $x=-a/2$ (we can also check by direct calculation that the 
derivative of $R$ is increasing on this interval  and it is  only zero at $x = -a/2$, 
so this point is the global minimum of $R$  over this interval).
Thus  for $x \in (-a+\delta, -\epsilon - \delta)$  we deduce that 
\begin{align*}
 R(x) 
  &\leq ( 1- \frac {\delta}{ a/2}) ( 1+ \frac {\delta}{-a/2}) (1+\epsilon) \\
	&\leq ( 1- \frac {2\epsilon +O(\epsilon^2)}{a^2})
	      ( 1- \frac {2\epsilon+O(\epsilon^2)}{a^2}) (1- \epsilon) \\
	&\leq ( 1- \frac {2\epsilon+O(\epsilon^2)} {a^2})^2 (1- \epsilon) \\
	&\leq  1 -  (\frac 4{a^2} -1) \epsilon + O(\epsilon^2). \qedhere
\end{align*} 
\end{proof} 

For $\epsilon <0$ the calculations are almost identical, just the logic of which 
terms are positive or negative changes. We state the corresponding estimates, 
but leave the verification to the reader. 

\begin{lemma} \label{R est 3}
If  $\epsilon <0$,  $a =1+\alpha$  and $x \in [0,1]$,  then
\[ 
R(x) \leq   1 +  ( 4  - \frac 1 {a^2}) \epsilon + O(\epsilon^2)
	= 1- (3+O(\alpha)+O(\epsilon))|\epsilon| < 1.
\]
\end{lemma} 

\begin{lemma} \label{R est 4}
If  $\epsilon <0$, $a =1+\alpha$ and $x\in [-a,0]$,  then
\[ 
R(x) \geq   1 -  (\frac 4{a^2} -1) \epsilon + O(\epsilon^2)
	  = 1 + (3 +O(\alpha) + O(\epsilon)) |\epsilon| > 1.
\]
\end{lemma} 

These lemmas show that  we  can take 
$|R(x) -1| \geq  \lambda |\epsilon|$ on $[-a,1]$ for 
any $\lambda < 3$, by taking $|\epsilon|$ sufficiently small, and 
$a$ sufficiently close to $1$.

We will not use the estimates derived above in precisely the form they were 
given. Instead, we will use rescaled versions, which we now state explicitly. 
Recall that we chose $n$ of the form $n=8m+1$, so that there were $n-1=8m  = 2N$ 
nodal intervals $\{I^n_j\}_1^{n-1}$, and that we defined intervals $\{G^n_k\}_1^N$
by taking groups of four adjacent nodal intervals. More precisely, 
$G^n_k = I^n_{4k-3} \cup I^n_{4k-2} \cup I^n_{4k-1} \cup I^n_{4k}$. Moreover, 
since $I^n_k = [r^n_k, r^n_{k+1}]$, 
where $\{r^n_j\}_1^n$ are the roots of $T_n$,
we have  $G^n_k = [r^n_{4k-3}, r^n_{k+1}]$, and the 
three interior roots of $G^n_k$ are $r^n_{4k-2}$, $r^n_{4k-1}$  and 
$r^n_{4k}$. 
See Figure \ref{Defn_Gnk}. 

\begin{figure}[htb]
\centerline{ 
\includegraphics[height=1.5in]{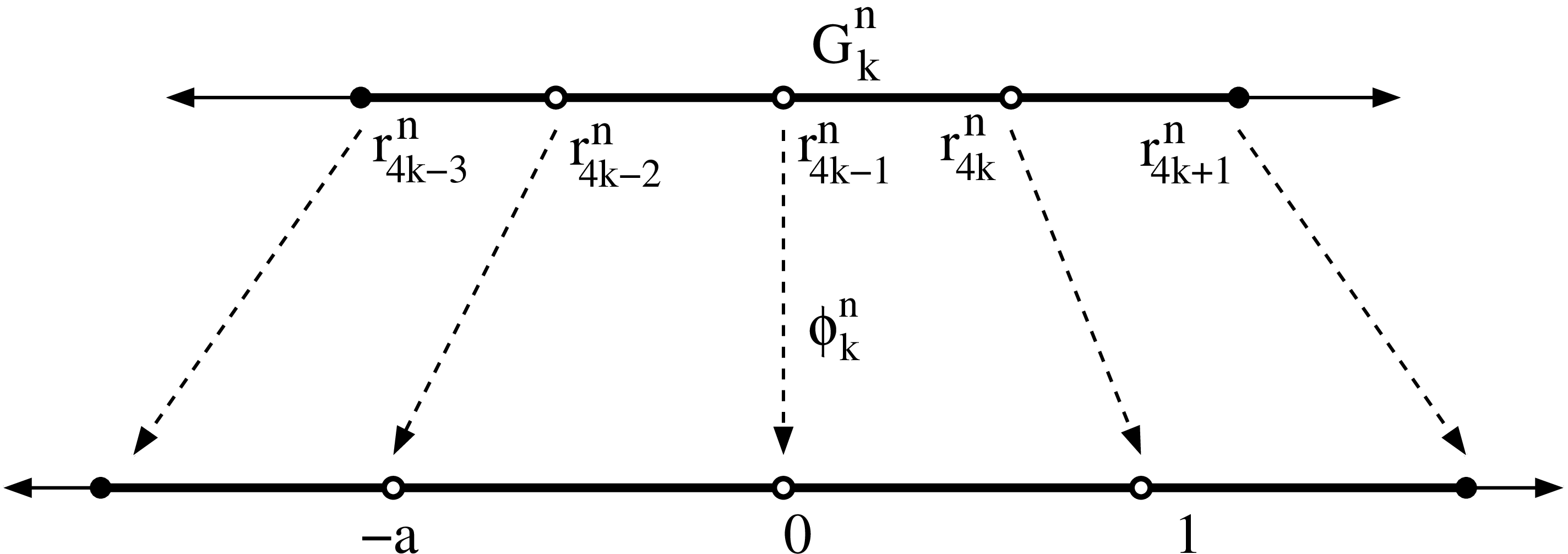}  
}
\caption{ \label{Defn_Gnk}
The definition of $G^n_k$ and the  linear map $\phi^n_k$. 
The white dots are the three roots interior to $G^n_k$; these map 
to $-a,0,1$ under $\phi^n_k$. 
}
\end{figure}

Define a linear map $\phi^n_k: \reals \to\reals$ by the 
conditions $\phi^n_k(r^n_{4k-1})=0$ and $\phi^n_k(r^n_{4k})=1$. 
In other words, we map the center and right interior roots of $G^n_k$ to 
$0$ and $1$ respectively. By Corollary \ref{K chains}, if $G^n_k$ is not too close 
to either $-1$ or $1$, then the four nodal intervals it contains are 
all approximately the same size and so  we get 
\[s=\phi^n_k(r^n_{4k-3}) \approx -2, \quad 
-a=\phi^n_k(r^n_{4k-2}) \approx -1, \quad  
t=\phi^n_k(r^n_{4k+1}) \approx 2.\]

Using the map $\phi^n_k$, a polynomial $p$ with  three roots inside $G^n_k$ 
corresponds to polynomial $q$ with three roots in $[s,t]$
by  $p(x) = q(\phi^n_k(x))/ |(\phi^n_k)'|^3$. If we take the ratio  of 
two such polynomials $p_1, p_2$, then the derivative factor cancels and we see 
that 
\[ \frac {p_1(x)}{p_2(x)} = \frac {q_1(\phi^n_k(x))}{q_2(\phi^n_k(x))} .\]
Thus the distortion function  $R^n_k$ for perturbations on $G^n_k$ is just a 
linear rescaling of the  3-point distortion function $R$  defined earlier in this section, 
i.e., $R^n_k(x) = R(\phi^n_k(x))$. With this, we can restate the results above 
for perturbations of roots in $G^n_k$. For example, the following are the 
rescaled versions of Corollaries  \ref{R est} and \ref{R est far}.

\begin{cor}  \label{R est rescaled}
If we perturb the interior roots of $G^n_k$ by a factor of $\epsilon$, and 
$|\epsilon|$ small enough,  then 
\begin{align*}
	R(x) \geq 1, &\text{ if }   \epsilon >0 \text{  and  } x \in (-\infty , r^n_{4k-2}] 
		       \cup [r^n_{4k-1},r^n_{4k}],\\
	R(x) \geq 1, &\text{ if }   \epsilon < 0 \text{  and  } 
		   x \in (r^n_{4k-2},r^n_{4k-1}] \cup [r^n_{4k},\infty), \\
	R(x) \leq 1, &\text{ if }   \epsilon <0 \text{  and  } 
	x \in (-\infty , r^n_{4k-2}] \cup [r^n_{4k-1},r^n_{4k}], \\
	R(x) \leq 1,  &\text{ if }   \epsilon > 0 \text{  and  }
	x \in (r^n_{4k-2},r^n_{4k-1}] \cup [r^n_{4k},\infty).
\end{align*} 
\end{cor} 

\begin{cor} \label{R est far rescaled}
	Suppose $a = 1+\alpha$ with $|\alpha|< 1/5$. For $ x \in [-1,1] \setminus
G^n_k$,  let 
\[
d= \frac{|x-r^n_{4k-2}|}{|r^n_{4k-1} - r^n_{4k}|} 
 = {|x-r^n_{4k-2}|}  \cdot |(\phi^n_k)'|
\]
be the distance between $x$ and the center 
root of $G^n_k$, normalized  by  $|I^n_{4k-1}|$. Then 
\[R(x) 
= 1 -  \frac { [2+O(\alpha)+O(\epsilon)] \epsilon   } 
 { (d+1)d(d-1) }  .
	 \]
\end{cor} 

Similarly,  Lemmas \ref{R est 1} to \ref{R est 4}
can be restated for perturbations  of the interior roots of $G^n_k$ 
by leaving the 
estimate for $R$ exactly the same  as before, 
and simply replacing the   intervals for $x$  by  new intervals 
obtained by replacing the  positions $\{-2, -a, 0, 1, 2\}$
by  the   points
$\{r^n_{4k-3}, r^n_{4k-2}$, $r^n_{4k-1}, r^n_{4k}, r^n_{4k+1} \}$.  
The  leftmost and rightmost points do not correspond exactly to 
$-2$ and $2$ under $\phi^n_k$, but 
only Corollary \ref{R est 0} makes use of these points, and in this 
case the corresponding  image points are so close that the estimate still
holds in the rescaled case.

\section{Bounding the extreme values}  \label{extreme sec} 

In  this  section, we show that small 
perturbations of the roots do not increase the extreme
values very much. This is needed in order  to show that our polynomial 
approximants can be taken to  be Lipschitz if the function $f$ being 
approximated is Lipschitz (as claimed in Theorem \ref{weier++}).
In a later section, we will  also use these estimates to bound
the size of the set where this perturbed Chebyshev polynomial 
is  close to zero, in order to prove our  approximants
have derivatives that diverge 
pointwise almost everywhere. See Lemma \ref{small set} and 
Corollary \ref{diverges}.

We first prove a special case ($n=3$) of the
Min-Max property of the Chebyshev polynomials, that 
was mentioned in Section \ref{length sec}.

\begin{lemma} \label{cubic}
Suppose $ r_1, r_2, r_3 \in I= [-1,1]$ and $p(x) = (x-r_1)(x-r_2)(x-r_3)$. 
Then $\max_I |p| \geq 1/4$ and the maximum is minimized by taking $r_1 =0$ and 
$r_2, r_3 = \pm \sqrt{3}/2$ (in other words, $p$ is a multiple of the 
Chebyshev polynomial $T_3$). 
\end{lemma} 

\begin{proof}
A direct calculation shows that the given  roots  satisfy
$\max_I|p| = 1/4$, so we only need to show that this is the best possible.
If $q$ minimizes the supremum of $|p|$ over $I$ among cubic monic polynomials 
with roots in $[-1,1]$ (a minimum exits by compactness), 
then let   $ \tilde q(x) = \frac 12 ( q(x) -q(-x))$. 
This is also cubic, monic  and  satisfies 
	\[ \sup_I \bigg| \frac {q(x) - q(-x)}{2} \bigg| \leq \sup_I | q(x) |.\]
This polynomial is clearly odd, so it has a root at $0$. 
If the other two roots were complex, they would
have to be both complex conjugates of each other
and also negatives of each other,  and 
hence both  would be zero. In this case $\tilde q(x) = x^3$ and  
$\tilde q(1) =1 > 1/4$, so this is not the minimum. 

Thus the minimizing  polynomial $q $ is  odd  with  
only  real roots:  $q(x) = x(x^2-r^2)$ for some $0<r\leq 1$.
If $ r < \sqrt{3}/2$
then $q(1) = (1-r)(1+r) = 1-r^2 > 1- 3/4 = 1/4$, so $q$ is not minimizing. 
If $r > \sqrt{3}/2$, then 
\[q(\frac{1}{2})  <  \frac 12  (\frac 1{2} -\frac {\sqrt{3}}{2} )( \frac 1{2} +\frac {\sqrt{3}}{2} )
=\frac 12 ( \frac 14  -\frac  3 4 ) = - \frac 14.\]
This is not minimizing either, so the optimal $r$ equals $\sqrt{3}/2$, as claimed. 
\end{proof}

\begin{lemma} \label{sup bound} 
For any $\lambda >1$ there is a 
 $\epsilon>0$, so that if   $\lVert y\rVert\leq \epsilon$  and 
$T_n(x,y)$ is the corresponding perturbation of the Chebyshev 
polynomial $T_n$, then 
\[    
	  \frac 1 \lambda   \sup_{G^n_m} |T_n(x)|  \leq 
	\sup_{G^n_m} | T_n(x,y)|  \leq  
	 \lambda  \sup_{G^n_m} |T_n(x)|.  \] 
\end{lemma} 

\begin{proof} 
By Corollary \ref{R est far rescaled},
the perturbations performed outside $G^n_m$ only 
multiply $T_n$ by a factor of $1 +O(\epsilon)$ inside 
$G^n_m$, so it suffices to show that the perturbations 
inside $G^n_m$ only change the supremum by a similar factor.

For brevity, let $I = G^n_m$.
Suppose we have three points $ r_1, r_2, r_3\in I$ and consider 
$h(x) = (x-r_1)(x-r_2)(x-r_3)$. 
By rescaling $[-1,1]$ to  $I$ in  Lemma \ref{cubic}, 
the supremum  of $|h|$ over  
$I$ is minimized (over choices of the roots in $I$)
by taking the degree 3 Chebyshev polynomial  on $I$,  and 
	in this case the 
minimum is $ (1/4)( |I|/2)^3 = |I|^3/32$. 
Thus   $\sup_I h \geq  |I|^3/32$.

Fix $\epsilon >0$ and let  $|\epsilon_j | \leq \epsilon |I|$ for $j=1,2,3$.
Define  a perturbation of $h$ 
\[ \widetilde h(x) = (x-r_1+\epsilon_1)(x-r_2+\epsilon_2)(x-r_3+\epsilon_3).\]
Then  if $\epsilon < 1/2$, 
\begin{align*}
	|h(x)-\widetilde h(x) | &=	|(x-r_1)(x-r_2)(x-r_3) - 
	(x-r_1 + \epsilon_1)(x-r_2+\epsilon_2)(x - r_3 + \epsilon_3)|   \\
	& \leq 
	|x-r_1||x-r_2||\epsilon_3| 
	+
	|x-r_1||x-r_3||\epsilon_2| 
	+
	|x-r_2||x-r_3||\epsilon_1| \\ 
	& \qquad \qquad 
	+
	|x-r_1||\epsilon_2||\epsilon_3| 
	+
	|x-r_2||\epsilon_1||\epsilon_3| 
	+
	|x-r_3||\epsilon_1||\epsilon_2| 
	+|\epsilon_1\epsilon_2 \epsilon_3| \\
	&  \leq 
	 3 |I|^3 \epsilon + 3 |I|^3 \epsilon^2 + |I|^3\epsilon^3 \\
	  &   \leq    4 |I|^3\epsilon = 128 (|I|^3/32) \epsilon  \\
	  &   \leq     128 \epsilon \cdot  \sup_I|h| 
	   \leq      \frac 12  \sup_I|h|, 
\end{align*} 
if $\epsilon < 1/256$.  Therefore 
$\sup_I|\widetilde h|  \leq  \sup_I |h| + \sup_I|h-\widetilde h|  \leq (1+ 128 \epsilon)\sup_I  |h|$
which is less than $\lambda \sup_I |h|$ if $\epsilon$ is small enough.
Similarly, 
$\sup_I|\widetilde h|  \geq  \sup_I |h| - \sup_I|h-\widetilde h|  \geq  \frac 1 \lambda \sup_I  |h|$
if $\epsilon$ is small enough, proving the lemma.
\end{proof}

\section{Estimating the  effect of interior perturbations} \label{interior sec}

In this section we start the proof of  Theorem \ref{delta estimate}.
We have to verify that when we perform a  3-point perturbation  
inside $G^n_m$, the integral of the Chebyshev polynomial 
over $G^n_m$ changes by a factor proportional to perturbation, 
and that  the effect on this integral 
of the perturbations in other intervals  $G^n_k$, $k \ne m$  is small
by comparison. 
Lemmas \ref{internal}  below and  Lemma \ref{external} in the next section
provide exactly these estimates. 
We end this section by verifying an earlier claim that the map defined 
by (\ref{defn f}) is a homeomorphism.

Also recall that in Section \ref{perturb sec} we introduced the notation $T_n(x,y)$ 
with $x \in [-1,1]$ and $y =(y_1, \dots, y_N) \in \reals^N$
to denote the $n$th Chebyshev polynomial $T_n(x)$ after we 
perturbed the three interior roots in  each $G^n_k$ by a factor $y_k$.
As before, let $\widetilde y_k \in \reals^N$ be the 
vector equal to $y_k$ in the $k$th coordinate and zero elsewhere. Thus 
$T_n(x,\widetilde y_k)$  corresponds to  perturbing only  the interior
roots of  $G^n_k$ and leaving all others fixed. 
Recall from (\ref{defn I}) that we defined   
\[ \Gint^n_m(y) =  \frac 1{|G^n_m|}\int_{G^n_m} T_n(x,y) dx.\]

\begin{lemma}  \label{internal} 
Suppose $t>0$,  $K \in \naturals$, $n \geq 2K$ and  $K \leq m \leq n-K$.
Suppose we  perturb the three interior roots of $G^n_m$ 
by a factor $y_m$ with $|y| \leq t$,   and we  leave 
all other roots of $T_n$ fixed. Let $ T_n(x,\widetilde y_m)$ denote 
the new polynomial obtained in this way.
Then  $|\Gint^n_k(\widetilde y_m)| \leq  \frac 76|G^n_m| $,  and 
 $\Gint^n_k(\widetilde y_m) $ is strictly  monotone. 
	Moreover, for $y_m >0$ we have  
\begin{align} \label{bound 1} 
	\Gint^n_m(\widetilde y_m) \geq 
	\Gint^n_m( 0)  + \frac{21}{20} y_m  = 
	\frac 1{|G^n_m|}   \int_{G^n_m}  T_n(x)dx  +  \frac{21}{20}   y_m 
\end{align} 
if $K$ is sufficiently large and $\epsilon$ is sufficiently small.  
For $y_m <0$ negative, we similarly get 
\begin{align} \label{bound 2} 
\Gint^n_m(\widetilde y_m) \leq
\Gint^n_m( 0)  - \frac{21}{20} y_m  = 
\frac 1{|G^n_m|}   \int_{G^n_m}  T_n(x)dx  -  \frac{21}{20}   y_m.
\end{align} 
\end{lemma}

The choice of $7/6$ and $21/20$ is only for convenience, to make some later 
arithmetic work out. Our proof of Theorem 1.1 only requires that these 
estimates hold with some constants strictly larger than $1$.  

\begin{proof}[Proof of Lemma \ref{internal}] 
The claim that  $|\Gint^n_k(\widetilde y_m)| \leq  \frac 76|G^n_m| $ holds 
because $|T_n(x,\widetilde y_m)| \leq  7/6$ if $\epsilon $ is small 
enough, by Lemma \ref{sup bound}. 

The monotonicity is also  easy to see.  Indeed, we simply want 
to verify that the situation shown in 
Figures \ref{Perturb_3point_plusminus}  and \ref{Perturb_3point_plotrange} is correct: 
as we increase the factor $y_m$ of the perturbation in $G^n_m$, the perturbed polynomials 
change strictly monotonically, and hence the same is true for their integrals over $G^n_m$. 
However, any two such  perturbed polynomials can be written 
as $P_1 P_0$ and $P_2 P_0$, where $P_1$ and $P_2$ are cubics  
corresponding to the three perturbed roots in $G^n_m$, and $P_0$ is the 
product over all other roots. Since $P_0$ has no zeros in the interior of 
$G^n_m$, it does not change sign there, and thus it suffices to show 
that $P_1$ and $P_2$ move monotonically as a function of the perturbation 
factor $y_m$. This was Lemma \ref{monotone}. 

Next we prove the quantitative bound (\ref{bound 1}); the proof 
of (\ref{bound 2})  is identical, except for obvious sign changes.
Recall that $G^n_m$ is divided into four nodal sub-intervals,
$I^n_{4m-3}, \dots I^n_{4m}$. Because the rescaled version of 
$G^n_m$ does not exactly match $[-2,2]$, the estimates
are easier for the two middle intervals $I^n_{4m-2}$ and 
$I^n_{4m-1}$, so we  deal with these first.
We also suppose the perturbation is by a positive factor $y_m >0$. 

On $I^n_{4m-1}$,  the rescaled version of  Lemma \ref{R est 1} says that
the positive node is multiplied by at least 
$ 1+  (4 - a^{-2})  y_m + O(y_m^2)$ everywhere on the subinterval, so the 
integral increases by an additive factor of at least 
\[
	[3 +O(\alpha) +O(y_m)] \cdot y_m \cdot \int_{I^n_{4m-1}} T_n(x) dx \geq 
[3 +O(\alpha) +O(y_m)]  \cdot\frac{2 y_m } \pi \cdot  |I^n_{4m-1}|,\]
where we have used  that $a = 1+\alpha$ 
and that (by Lemma  \ref{int lower bound})
the integral of $T_n$ over a nodal interval $I^n_k$ is at least 
	$(2/\pi)|I^n_k|$.
Similarly,  by the rescaled version of Lemma \ref{R est 2}, the  absolute 
value of the negative node on $I^n_{4m-2}$ 
becomes smaller, and its integral 
increases by a positive additive factor of size at least  
\[   [(4 a^{-2}-1) y_m  +O(y_m^2)] \frac 2 \pi  |I^n_{4m-2}| 
 \geq  [3 +O(\alpha)+O(y_m)] \cdot \frac {2 y_m} \pi \cdot  |I^n_{4m-2}|.\]

Next we deal with the outer sub-intervals, namely $I^n_{4m-3}$
and $I^n_{4m}$. 
Let $J \subset I^n_{4m-3}$ denote the part of 
$I^n_{4m-3}$ that lands in $[-2,a]$ when we rescale 
$G^n_m$ as described above (the left root maps to $a$, 
the center root is mapped to $0$, 
and the right root maps to $1$).
Possibly $I^n_{4m-3} \setminus J$ is empty if  $I^n_{4m-3}$ maps into $[-2,a]$.
The size of this ``leftover''  interval is small if all four  subintervals 
of $G^n_m$ are about the same size, which happens if $G^n_m$ is 
not too near $-1$ or $+1$ by Lemma \ref{length growth}. This is 
where we use the assumption that $K \leq m \leq n-K$ for some large $K$.

More precisely, for any $\eta  > 0$, Lemma \ref{length growth}  says all 
the nodal intervals in $G^n_m$ have the same length up to a
multiplicative factor of $ 1- \eta$, if $K$ is large enough. Hence 
we may assume that they all have length  
at least $(1-\eta)|G^n_m|/4$. 
Thus the part of $I^n_{4m-3}$ that is not in $J$ has 
length at most $O(\eta|I^n_{4m-3} |)$, and it maps to an 
interval of length $O(\eta)$ (possibly empty) to the left of $-2$ under the rescaling.
On this part of $I^n_{4m-3}$, Corollary \ref{R est rescaled} says that 
the perturbed polynomial is larger than $T_n$, so the integral  over 
this segment increases under the perturbation, since 
$T_n$ is positive on $I^n_{4m-3}$.

On $J$, we use  the rescaled version of Corollary \ref{R est 0}, which says  the 
3-point perturbation makes $T_n$ larger on $J$ 
 by  a multiplicative  factor of  
$1+ (\frac 13 +O(\alpha) + O(y_m))y_m$. 
Thus the integral $\int_J T_n(x) dx $ will increase 
by an additive factor of 
\begin{align*} 
[\frac 13 +O(\alpha) + O(y_m)]y_m \int_J T_n(x) dx
	&\geq  [\frac 13 +O(\alpha) + O(y_m)] \cdot \frac {2 y_m }\pi |J| \\
	&\geq  [\frac 13 +O(\alpha) + O(y_m)] \cdot \frac {2 y_m }\pi (1-\eta)|I^n_{4m-3}| ,
\end{align*} 
where we have used  Lemma \ref{int lower bound} again to give a lower 
bound for area of a node  in terms of the length of the base interval. 
Since we have already shown that the integral over $I^n_{4m-3} \setminus J$
changes in the same (positive) direction, the integral over all 
of $I^n_{4m-3}$ changes by at least the bound given for $J$.
By a very similar argument, 
the integral of the negative node in  $I^n_{4m}$   is made smaller 
(in absolute value), by at least an additive factor of the same size. 

Thus the increase in the integral over
all of $G^n_m$ is at least
\begin{align*}
&[ \frac 2 3 + 6 +O(\alpha)+O(y_m)]  \frac {2 y_m} \pi  \cdot (1-\eta) |G^n_m|/4 \\
&\qquad  \qquad \qquad \geq ( \frac   {10}{3 \pi} +O(\alpha)+O(y_m)) (1-\eta)   y_m  |G^n_m|
	\geq   \frac {21}{20} \cdot   y_m |G^n_m|,
\end{align*} 
if $\eta$  and $y_m$ are small enough (note $10/(3 \pi) \approx  1.06103 > 21/20 
= 1.05$ and $|\alpha| \leq \eta$).
\end{proof} 

Figure \ref{Perturb_3point_plotrange} illustrates a computation of the 
change in the integral in a special case, and indicates the estimate 
of $y_n |G^n_m|$ 
in the previous lemma is within a factor of three of being sharp.

\begin{figure}[htb]
\centerline{ 
\includegraphics[height=2.5in]{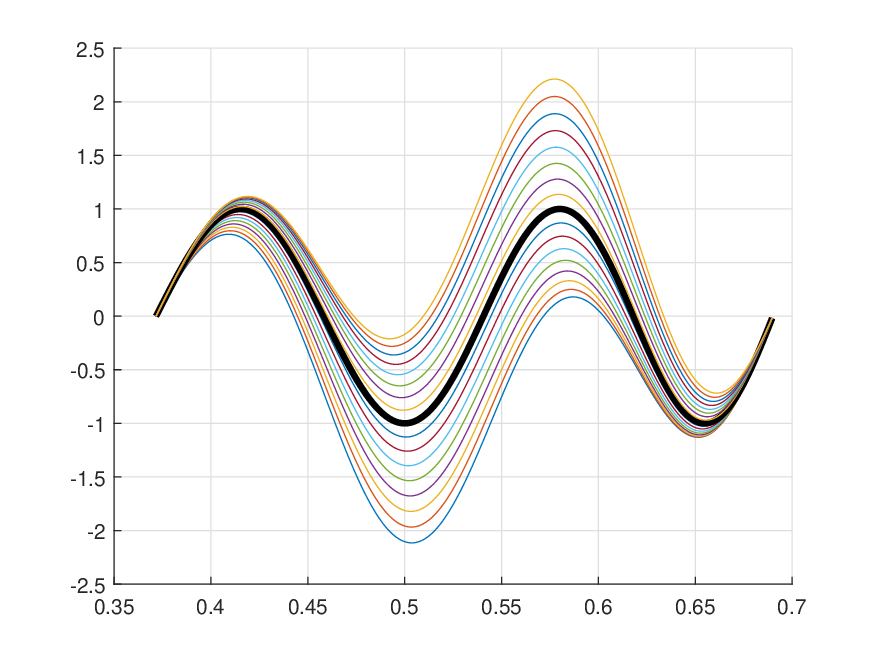}  
\includegraphics[height=2.5in]{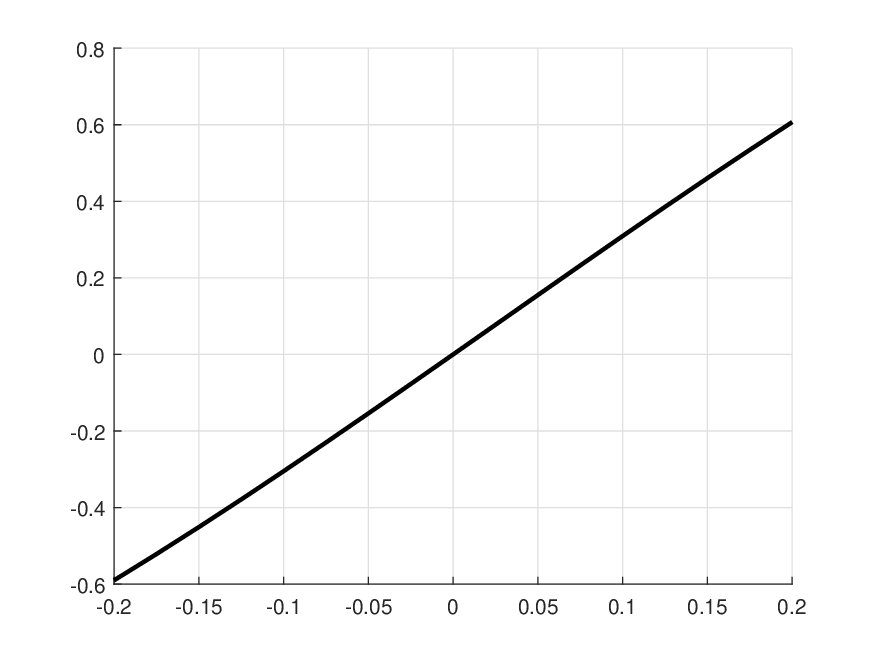}  
}
\caption{ \label{Perturb_3point_plotrange}
On the left  are 3-point perturbations for $y = -.2, -.175, \dots, .2$. 
The original Chebyshev polynomial is  highlighted.
On the right is a plot of  $\Gint^n_m(y) -\Gint^n_m(0)$ for these perturbations.
It has slope close to $3$, while our proof  shows the slope is  $\geq 1$.
This example was taken with $n=33$ and $m=5$.
}
\end{figure}

\begin{cor} \label{f is homeo}
If $t>0$ is small enough and  $n, K \in \naturals$ are both large enough, then 
 the map $f$ as defined in Equation (\ref{defn f}) is a homeomorphism from 
$Q^{N}_{t}$ to a cube  $Q'$ containing $Q^N_{t}$. 
\end{cor}

\begin{proof}
By definition, the $k$th coordinate of $f$ depends only on 
the $k$th coordinate of $y$, so the image is a cube, i.e., 
a product of compact intervals.
By Lemma \ref{internal}, every coordinate of $f$ is a monotone 
function of the $k$th coordinate of $y$. Hence $f$ is injective,
and thus a homeomorphism. 
Finally,   (\ref{bound 1}) and (\ref{bound 2})   imply 
the image under
$f_k$ of $[0, t]$ has length at least $\frac {21}{20} t$. 
Moreover, by Corollary \ref{int -> 0}  we know $\Gint^n_m(0)  < \eta/3$ 
if  $n \geq 6/\sqrt{\eta}$. (Recall that $\eta>0$ is our
upper bound for the length ratio between 
nodal intervals inside  $G^n_m$.
This can be taken as close to zero as we wish by taking $K$ 
large enough.)
If  $\eta \leq t/10$, and  $n \geq 6 /\sqrt{\eta}$, 
then $|\Gint^n_m(0)|  <t/20$. Therefore  $f_k([0, t])$ contains 
$[\Gint^n_m(0), t]$.  The same argument shows
 $f_k([- t,0]) \supset [-t ,\Gint^n_m(0),]$, and 
this implies $f(Q^N_t)$  contains $Q^{N}_t$.
\end{proof} 

\section{Estimating the  effect of exterior  perturbations} \label{exterior sec}

Next we see how perturbations of the roots outside $G^n_m$ affect 
$T_n$ inside this interval. This will complete the proof of 
Theorem \ref{delta estimate}.

\begin{lemma}  \label{external} 
Suppose  $y \in Q^N_t$. If $t>0$ is small enough  and $K \in \naturals$
is large enough, then for $K <  m \leq N-K$ we have 
\[
	|\Gint^n_m(\widetilde y_m) - \Gint^n_m(y) | = 
	\bigg|  \frac 1{|G^n_m|} \int_{G^n_m}  T_n(x,\widetilde y_m) -T_n(x,y)  dx  
	\bigg|\leq   
\frac {t}{2}.  \]
\end{lemma} 

\begin{proof}
For a fixed $m$ we want to estimate the contribution of all 
perturbations in $G^n_k$ for $k\ne m$ to the distortion function 
in $G^n_m$. 
Suppose $M \in \naturals$ (we will fix a value below). 
We divide  the intervals $\{G^n_k\}_{k \ne m}$  into two groups according to 
whether $|k-m|\leq M$ or $|k-m|> M$.  The second case  (the more distant 
intervals) is easier, and we deal with it first.

Suppose $I$ and $J$ are the two intervals formed from 
$G^n_k$ and $G^n_m$ after rescaling the real line 
	so that the 
center and right interior roots of $G^n_k$ map to  $0$ and $1$. 
Then $I$  is approximately $[-2,2]$ if $K < k \leq N-K$,
by Corollary \ref{length growth}. 
With this normalization, 
Lemma \ref{R est far rescaled} says that 
the  distortion $R(x)$  at a point $x \in J$   of a perturbation  
by a factor  $t$ in $I$ is  at most 
\begin{align}\label{distortion bound} 
1+ \frac{ 2 t  +O(\alpha t) + O(t^2)}{(x-1)x(x+1)}.
\end{align} 
Recall that here $\alpha = a-1$, so  $|\alpha|$ is as small 
as we wish by taking $K$ large enough.

Therefore
the  distortion bound  on $G^n_m$ for perturbations in $G^n_k$  is largest
at the endpoint of $G^n_m$ closest to $G^n_k$, and is bounded 
by the formula above, except that $x$ is replaced by the distance 
from $G^n_m$ to the center root of $G^n_k$, divided by the distance 
between the center and right roots of $G^n_k$. 
We can simplify this a little by replacing the distance between 
$G^n_m$ and the center of $G^n_k$ by the distance from $G^n_m$ to 
$G^n_k$. This is smaller, so gives a slightly larger bound.
By Corollary \ref{ratio upper bound}, and because each $G^n_j$ is 
made up  of four nodal intervals, we have 
\[ \frac {|G^n_k|}{\dist(G^n_k,G^n_m)} \leq \frac {16}{4|k-m|-1} 
\leq \frac {16}{4|k-m|-4}  = \frac {4}{|k-m|-1}.\]

Recall from calculus that  for $x >0$,  $1+x = \exp(\log(1+x)) \leq \exp(x) $. 
	Also, if $f$ is decreasing on $[M, \infty)$ then 
	$\sum_{j=M}^\infty f(j) \leq \int_{M-1}^\infty f(x) x$.
By definition, the distortion functions for perturbing distinct 
sets of roots multiply to give the total distortion function, 
so we see that the total distortion on $G^n_m$ due to perturbations 
in all  $G^n_k$ with $|k-m| >  M$ is bounded by the product 
\begin{align*} 
	&\prod_{k:|k-m|>M} \Big( 1 + \frac{C t}{(|k-m|-1)^3 }  \Big)
= 
\prod_{j \geq M} \Big( 1 + \frac{C t}{j^3}  \Big) \\
&\qquad \qquad  
\leq   \exp \bigg( \sum_{j\geq M} \frac {C t}{j^3} \bigg) 
\leq   \exp \bigg( C t \int_{M-1}^\infty x^{-3} dx  \bigg) 
\leq   \exp \bigg(  \frac {C t}{ 2  (M-1)^{2} }   \bigg) .
\end{align*} 
The final term is less than  $1 +   t/M$ if  $M$ is large enough.
Thus the distant intervals contribute almost no distortion.
	 
Next we consider the distortion due to ``nearby'' intervals, i.e., 
the effect on $G^n_m$ of perturbations in $G^n_k$ with $|m-k| \leq M$.
This is more delicate than the ``distant'' intervals, and 
getting the  first few terms (corresponding to 
intervals adjacent and nearly adjacent to $G^n_m$) to be small enough is 
one reason why we have used 3-point perturbations, instead of the simpler 2-point 
perturbations.

Fix $\eta >0$. If $K$ is large enough (depending on 
$\eta$ and $M$), then by Lemma \ref{length growth}  we can 
assume all the  nodal intervals contained in intervals  $G^n_k$ with 
$|m-k| \leq M$ have lengths 
within a factor of $1 -\eta$ of each other.  To simplify calculations 
we normalize  $G^n_m$ as before, with the center and right-hand 
interior roots mapping to  $0$ and $1$ respectively, 
and $G^n_m$ maps to approximately $[-2,2]$.
Thus the  renormalized nodal intervals have  length 
approximately $1$ (within a multiplicative factor of size $1+\eta$).
Since $ 1-\eta \leq (1+\eta)^-1$ we  have that 
$1-\eta \leq |I|/|J| \leq 1+\eta$ for any two nodal intervals 
$I, J \subset G^n_k$. 

With these assumptions,  if $|k-m| \leq M$, then 
the distance between the center of $G^n_m$ and 
a $G^n_k$  is at least $x=4(1-\eta)|k-m|$. 
If $G^n_k = G^n_{m+j}$  is to the right of $G^n_m$, 
then $G^n_{m+j}$  is approximately the interval $[4j-2, 4j+2]$ for 
$j= 1, 2, \dots$, M  (with error at most $\eta$)
and the maximum of our bound for the distortion  on $G^n_m$
by a perturbation in  $G^n_{m+j}$ occurs 
at the right endpoint of $G^n_m$, since this is the endpoint of $G^n_m$ 
that is closest to $G^n_{m+j}$.
Our distortion bound is smallest at the left endpoint of $G^n_m$, which 
is the furthest point of $G^n_m$ from $G^n_{m+j}$. 
At this endpoint, our estimates say the distortion  is at most 
\[ 1 + \frac {[2 +O(\eta)+O(t)]t}{(4j-3)(4j-2)(4j-1)}.\]
The smallest size of our estimate occurs at the endpoint
of $G^n_m$ farthest from $G^n_{m+j}$ and equals 
\[ 1 + \frac {[2 +O(\eta)+O(t)]t}{(4j+1)(4j+2)(4j+3)}
 =  1 + \frac {[2 +O(\eta)+O(t)]t}{(4(j+1) -3)(4(j+1)-2)(4(j+1)-1)}.\]
Similar estimates hold for perturbations in intervals to the left
of $G^n_m$, i.e.,  in $G^n_{m-j}$ for $j=1,2, \dots, M$.

Below, we will want to estimate the product of these terms over  
all $k = m-M , \dots m+M$, except for $k=m$. We can get a slightly better 
estimate by pairing  symmetrically placed terms of the form $ m \pm j$, 
and we take advantage of this as follows. For the moment,  we consider 
only the denominators in the bounds above. 

It is easy to check from the explicit formula  that the rational function giving the 
distortion bound  due to perturbations in $G^n_{m+j}$ is convex 
as a function of $x$ on $G^n_m$ (e.g., its partial fraction expansion is 
a sum of three convex terms on this interval). Similarly,  the distortion bound 
for perturbations in $G^n_{m-j}$ is convex on $G^n_m$. 
Therefore, the sum of these bounds is convex on this 
interval, and thus the sum takes its maximum value at one of the endpoints. 
See Figure \ref{ConvexSum}.

\begin{figure}[htb]
\centerline{
\includegraphics[height=1.5in]{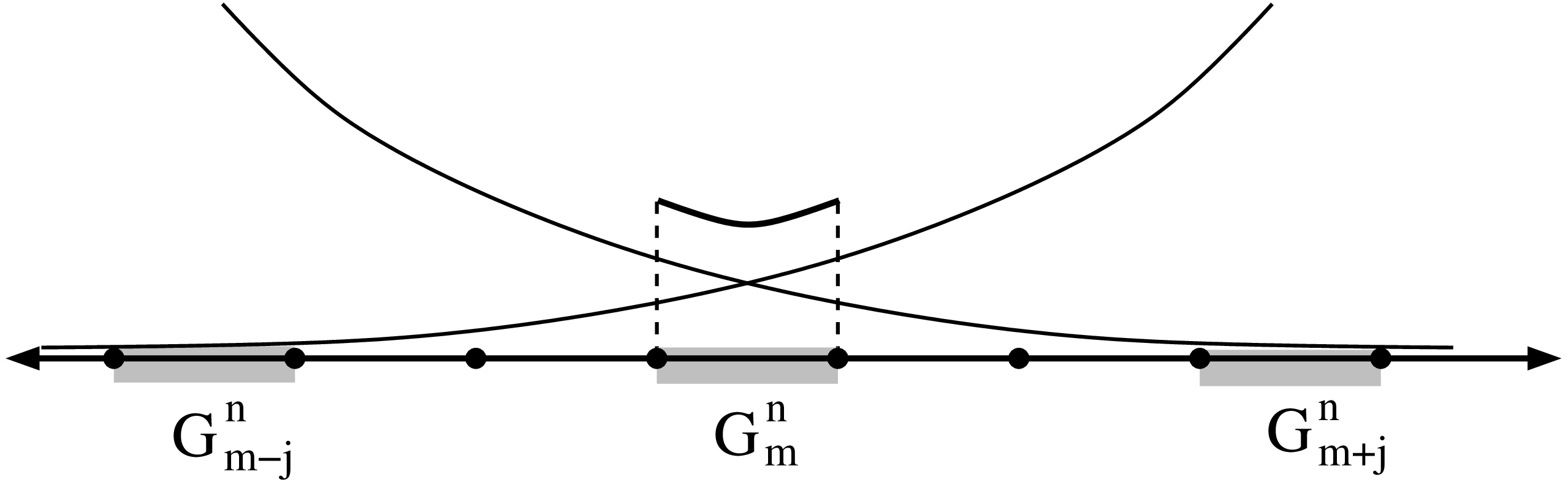}
}
\caption{ \label{ConvexSum}
The sum of the distortion bounds corresponding to $G^n_{m-j}$ and 
$G^n_{m+j}$ is convex on $G^n_m$ and hence bounded on 
$G^n_m$ by its values at the endpoints. 
}
\end{figure}

Using the elementary observation that 
\[(1+x)(1+y) \leq \exp(\log((1+x)(1+y)))
\leq \exp(x+y),\]
we can bound the product of the distortion bounds  for 
the distortions in both $G^n_{m+j}$ and $G^n_{m-j}$ 
by 
\[ \exp\Big(   \frac {1} {(4j-3)(4j-2)(4j-1)}
   + \frac {1} {(4(j+1) -3)(4(j+1)-2)(4(j+1)-1)} \Big) .\]
   Taking the product of distortions for all $j$ is thus 
   bounded by the exponential of the corresponding sum  of these
   fractions over $j=1,\dots, M$.
   Because  the second fraction above is the same as the first, but with 
   $j$ replaced by $j+1$,  
   each fraction is repeated twice, except 
   for the first and last.  Thus  we get the upper bound for the sum 
\[  \frac 1{1 \cdot 2 \cdot 3} + 2 \sum_{j=2}^M
   \frac {1} {(4j+1)(4j+2)(4j+3)}   \leq 
  \frac 1{6} + \frac 2{210} + \frac{2} {990} + \frac{2}{2730} + \dots.\]
The sum becomes larger by replacing $M$ by $\infty$, and we can 
bound the infinite sum (that clearly converges) by computing a finite number
$S$ of terms and bounding the remaining  tail by the estimate 
\[ 
\sum_{j=S}^\infty \frac {1} {(4j+1)(4j+2)(4j+3)}   
\leq 
\sum_{j=S}^\infty \frac {1} {(4j+1)^3}   
\leq 
       \sum_{j=4S+1}^\infty \frac {1} {j^3}   
\leq  \int_{4S}^\infty x^{-3} dx 
\leq   \frac 1{32S^2}.
\]
Taking $S=100$ gives the upper bound $.1799 < 1/5$.

Using this (and the fact $1+x \leq \exp(x)$), 
the distortion bound on $G^n_m$ due to 
perturbations within $M$ steps of $G^n_m$   is bounded  by
\begin{align*}
	&\prod_{k: 0<|k-m|\leq M} \bigg( 1 + 
\frac{ (2+ O(\eta)+O(t)) t}{ |(|k-m|+1)|k-m|(|k-m|-1)|}     \bigg)\\
	& \qquad \qquad \leq 
 \exp\bigg( \sum_{k: 0<|k-m|\leq M}  
\frac{ (2+ O(\eta)+O(t)) t}{ |(|k-m|+1)|k-m|(|k-m|-1)|}    \bigg)\\
	& \qquad \qquad\leq
	\exp\Big( \frac{ (2+ O(\eta)+O(t)) t}{ 5 }   \Big).
\end{align*} 
By taking  $\eta$ and $t$ small enough,  we can make this 
less than $1+ \lambda t$ for any $\lambda > 2/5$. We previously 
proved the distortion contributed by the distant intervals could be taken 
to be less than $1+ t/M$, so by taking $M$ large enough
(say $M \geq 10$), 
the total distortion from perturbations outside $G^n_m$ 
is less than $1 +  3 t/7$. Thus  on $G^n_m$ we have 
\[    (1- \frac {3 t}7 )  T_n(x, \widetilde y_m)
\leq  T_n(x,y) \leq (1+  \frac {3 t}7 ) T_n(x, \widetilde y_m),
\]
and   hence integrating over $G^n_m$,  
\[   (1-\frac {3 t} 7  ) \Gint^n_m(\widetilde y_m)
\leq  \Gint^n_m(y) \leq (1+ \frac {3 t}7 ) \Gint^n_m(\widetilde y_m).
\]
If we take $t$ small enough that
$|\Gint^n_m(\widetilde y_m)| \leq \frac 76 |G^n_m| $, then this implies 
\[ |\Gint^n_m(y) - \Gint^n_m(\widetilde y_m)|
\leq  \frac  {3 t}7  |\Gint^n_m(\widetilde y_m )| 
\leq  \frac {3 t} 7  \cdot  \frac 76 |G^n_m|   = \frac  t 2 |G^n_m|,
\]
as desired.
\end{proof} 

This completes the proof of  Theorem \ref{delta estimate}. In the next 
section we use it to prove our main result,  Theorem   \ref{weier++}: 
polynomials with all critical points in a compact interval $I$ are 
dense in $C_\reals(I)$.

\section{Applying Brouwer's fixed point theorem} \label{fixed pt sec} 

Recall that $Q^n_t= [-t,t]^n$, that
$\lVert(x_1, \dots , x_n)\rVert = \max_{1\leq k \leq n} |x_k|$ 
denotes the supremum norm on $\reals^n$, and 
that $\lVert f-g\rVert_Q = \sup_{x \in Q} |f(x)-g(x)|$.
By Brouwer's fixed point theorem \cite{Brouwer}, any continuous map of 
$Q^n_t$ into itself has a fixed point. There are now 
various short proofs of this result, e.g., 
\cite{MR505523}. 

\begin{lemma}  \label{topology lemma 2} 
Suppose that $I_t = [-t,t]$,   and 
for $k=1, \dots, n$, that   $J_k \subset \reals$  is a compact interval 
that contains $I_t$.  Let $Q = Q^n_t= \prod_{k=1}^{N} I_t$
and $Q' = \prod_{k=1}^{N} J_k$. Suppose  $f= (f_1, \dots, f_{N})$ 
is a homeomorphism from $Q$ to $Q'$, and  that $g:Q \to \reals^N$ 
is a continuous map such that   $ \lVert f-g \rVert _Q \leq t/2$. 
Then $Q^n_{t/2} \subset g(Q)$. 
\end{lemma} 

\begin{proof} 
Suppose  $a\in Q^n_{t/2}$. We want to show there is $\widehat x \in Q$ so that 
$g(\widehat x) = a$.  
For $x \in Q^n_t$, define $F(x) = f^{-1}(a+f(x)- g(x))$. 
By assumption, for 
$x \in Q$, 
\[ \lVert a + f(x) - g(x)\rVert 
 \leq  \lVert a \rVert  +  \lVert f(x) - g(x)\rVert 
  \leq t/2 + t/2 = t,\]
	so $a+f(x)-g(x) \in Q' = f(Q)$. Thus $F(x)$ is well defined 
	and $F(Q)  \subset  Q$.
Since $f$ is a homeomorphism, $f^{-1}$ is continuous, and hence 
$F$ is continuous. 
Thus Brouwer's fixed 
point theorem  implies $F$ has a fixed point $\widehat x \in Q$. At this point 
\begin{align*} 
	f^{-1}( a +f(\widehat x) -g(\widehat x))  &=  \widehat x \\
	a +f(\widehat x) -g(\widehat x)  &=  f(\widehat x) \\
	 a &=  g(\widehat x). \qedhere
\end{align*} 	
\end{proof}

We want to apply this to the functions $f$ and $g$ defined in 
Equations (\ref{defn f}) and (\ref{defn g}). Lemma  \ref{internal} showed that 
$f$ is a homeomorphism of $Q=Q^N_{t}$  onto a cube $Q'$ containing 
$Q$, at least if $t>0$ is small enough. Theorem \ref{delta estimate} shows that 
$\lVert f-g \rVert_Q  \leq t/2$. Thus we are in a 
position to apply Lemma \ref{topology lemma 2} to these functions,  in order to 
prove the following  (slightly stronger) version of Theorem \ref{weier++}.

\begin{cor} \label{Lip approx} 
Any Lipschitz function $F$ on $[-1,1]$ can 
be uniformly approximated  to within $1/n$
by a   polynomial $P$ of degree $O(n)$ 
with all its  (real or complex) critical points 
in $[-1,1]$.  Moreover, $P$ and $F$ agree at  both
endpoints of every interval $G^n_m$, except for a 
uniformly bounded number at the beginning and end of 
$[-1,1]$. 
If $F$ is $A$-Lipschitz, then we can choose 
$P$  to be  $CA$-Lipschitz
with a constant $C$ that is independent of $F$ and $n$.
\end{cor} 

\begin{proof} 
We claim it suffices prove that 
every $t$-Lipschitz function can be uniformly approximated, 
for some positive value of $t$.  
To see this, note that if $h$ is $A$-Lipschitz, 
then $\widetilde h=  (t/A)  \cdot h$
is $t$-Lipschitz,
and if $\widetilde p$ approximates $\widetilde h$ to 
within $\epsilon t/A$ then $p=(A/t) \cdot \widetilde p$ 
approximates $h$ to within $\epsilon$.  
This proves the claim.
Also note that $p$ is $(2A/t)$-Lipschitz if $\widetilde p$ 
is $2$-Lipschitz, which will be the case below. Thus we can 
	take $C = 2/t$ in the statement of Theorem \ref{weier++}.

Choose $t$  so that   Theorem \ref{delta estimate} holds, 
i.e., so that $\lVert f-g\rVert_{Q^N_{2t}} \leq t$. 
Suppose $G^n_k = [s,t]$ and set  
\[a_k =  \Delta(F,G^n_k)/|G^n_k| = (F(t)-F(s))/(t-s) .\]
Then $a \in Q^N_t$ since $F$ is $t$-Lipschitz.
Then by Lemma \ref{topology lemma 2},
for any $n\in 8\naturals+1$ sufficiently large,
there is  a $y\in  Q_t^N$, so that  the  perturbed
Chebyshev polynomial $T_n(x,y)$  satisfies 
	\[ \Gint^n_k(y) = \frac 1{|G^n_k|} \int_{G^n_k} T_n(x,  y) dx  = a_k,
	  \text{ for all } k=1,\dots N.  \]
Thus the  anti-derivative
	\[ P(x) =  F(0)+  \int_{0}^x T_n(t,y) dt \]
is a polynomial of degree $n$ that has all its 
critical points in $[-1,1]$ and  satisfies 
$ P(x) = F(x) $ at each endpoint of any 
$G^n_k$ (except possibly for a bounded number $K$ at each 
end of $[-1,1]$.

By Lemma \ref{sup bound}, 
$|P'(x) | = |T_n(x,y)|$ is bounded by $2$ if $t$ 
is small enough; thus $P$ is $2$-Lipschitz.   
Thus $|P(x)-f(x)| \leq (2+t)|G^n_k|/2$ on 
$|G^n_k|$, except for $K$ intervals  near each end, where 
we get the bound $ K(2+t)\max_k |G^n_k|$. 
Since $\max_k |G^n_k| \leq 4 \pi/n \to 0$  
(Lemma \ref{length estimate}), we 
see that $P$ approximates $F$ to within $O(1/n)$ on $[-1,1]$.
\end{proof} 

This completes the proof of Theorem \ref{weier++}. 

\section{Weak-$\ast$ convergence: Proof of Theorem \ref{weak limit} } 
\label{weak sec}

In this section, 
we prove  Theorem \ref{weak limit}, that every bounded, measurable, real-valued  
function on $[-1,1]$ is the weak-$\ast$ limit of  real polynomials with only real 
critical points. 
Recall  that a sequence $\{ f_n\} \subset L^\infty$ converges 
weak-$\ast$ to $f \in L^\infty$ if for every $g \in L^1$, 
$ \int_I f_n g dx \to \int_I f g dx$.
The definitions and results  we quote below can be found 
in standard texts such as \cite{MR1681462}.

\begin{proof}[Proof of Theorem \ref{weak limit}]
It suffices to prove this for functions on $[-1,1]$.
	Fix a real-valued  $f \in L^\infty([-1,1])$.
We want to find polynomials $\{P_n\}$
	that are uniformly bounded on $ [-1,1]$ and so that 
	$\int gP_n \to \int gf $ for any $g \in L^1([-1,1])$. 
Let $\{G^n_k\}$ be the 
	partition of $[-1,1]$ into unions of four adjacent  Chebyshev nodal 
intervals, as in the proof of Theorem \ref{Lip approx}. Using 
that theorem, there is a  $K \in \naturals$ and a real-valued polynomial $P_n$ so that 
$\lVert P_n\rVert_\infty \leq C \lVert f\rVert_\infty$ and 
\begin{align} \label{equality} 
\int_{G^n_k} P_n = \int_{G^n_k} f
\end{align} 
 for every $k=K, \dots, (n/4)-K$.
The union of the $2K$ intervals $G^n_k$  where this estimate does not 
hold have total length tending to zero as $n$ increases to $\infty$. 
Fix  $g \in L^1([-1,1])$ and note that 
both $\int gf$ and $\int gP_n$ tend to zero  over these intervals,
so we can restrict attention to the union $I \subset [-1,1]$
of sub-intervals where  (\ref{equality})  does hold.

For $M < \infty$ define  $g_M$ by 
$g_M =g$ on $\{x\in I: |g(x)|<M\}$ and  $g_M =0$ elsewhere. Since 
$|g-g_M| \leq g \in L^1$ and $g_M \to g$ pointwise almost everywhere, 
the Lebesgue dominated convergence theorem implies $\lVert g-g_M\rVert_1 \to 0$
as $M \nearrow \infty$. So given 
any $\epsilon>0$,  we can choose $M$ so large that $\lVert g-g_M\rVert_1 \leq \epsilon$.
Since $\lVert P_m\rVert_\infty \leq C \lVert f\rVert_\infty$, we have 
\[ \int_I P_n g - \int_I f g  = \int_I (P_n-f) g_M + \int_I (P_n-f) (g-g_M) 
= \int_I (P_n-f) g_M +  O( \epsilon\lVert f\rVert_\infty).\]
Therefore, it is enough to show  $\int_I (P_n-f)g  \to 0$
as   $n \nearrow \infty$.

Let $\psi_n$ be  a step function approximation to $g_M$ that 
is constant on the segments $\{ G^n_k\}$ and so that 
$\lVert g_M -\psi_n\rVert_1 \to 0$ as $n \nearrow \infty$.  
Note that   $\int_I (f-P_n)\psi_n=0$, since 
the integral of $f-P_n$ is zero on each sub-interval $G^n_k$ 
where $\psi_n $ is constant.  Therefore, 
\begin{align*} 
	|\int_I (P_n -f) g_M | 
	&=
	|\int_I  P_n (g_M -\psi_n) + \int_I (P_n -f)\psi_n + \int_I (\psi_n-g_M)f|  \\
	&\leq 
	\lVert P_n \rVert_\infty \lVert g_ M-\psi_n\rVert_1 +  \lVert f\rVert_\infty \lVert\psi_n-g_M\rVert_1  \\
	&\leq (C+1)\lVert f\rVert_\infty \lVert g_M - \psi_n\rVert_1,
\end{align*} 
and the last term tends to zero as $n$ increases. 
Thus $\int_I P_n g \to \int_I fg$ 
for any $g \in L^1(I)$,  and hence $P_n \to f$ weak-$\ast$.

The final step is to show that we cannot take  $C=1$ in 
Theorem \ref{weak limit}. Suppose  we could. 
Define  $f =1$ on $[0,1]$ and  $f=0$  on 
$[-1,0)$, and suppose  that $p_n$  are polynomials 
with only reals zeros, that  $\lVert p_n\rVert_\infty \leq 1$, and  that 
$p_n$ converges  weak-$\ast$ to $f$.
By weak convergence, $\int_0^1 p_n dx \to 1= \int_0^1 f dx$,  but 
for any $\epsilon >0$, 
\begin{align*} 
	\int_0^1 p_n dx 
	&\leq |\{x\in [0,1]: p_n(x) > 1-\epsilon\}| 
	+ (1- \epsilon) |\{x\in [0,1]: p_n(x)  \leq  1-\epsilon\}| \\
	&= 1- \epsilon |\{x\in [0,1]: p_n(x)  \leq  1-\epsilon\}|.
\end{align*} 
For $\epsilon>0$ fixed, the
only way the right-hand side can tend to $1$ is if  
\[|\{x\in [0,1]: p(x)  \leq  1-\epsilon\}| 
 = |\{x\in [0,1]: |p(x) -f(x)| \geq \epsilon\}| \to 0 .\]
Thus $p_n$ converges to $f$ in measure on $[0,1]$.
By the Clunie-Kuijlaars theorem discussed in 
the introduction (Corollary 1.3, \cite{MR1294323}), 
$f$ must be an entire function.  But $f$ is discontinuous, 
a contradiction.
Therefore,  $C=1$ is impossible. 
\end{proof} 

\begin{ques}
What is the optimal value of $C$ in Theorem  \ref{weak limit}? 
\end{ques} 

\section{Divergence almost everywhere}  \label{diverges ae}

In this section we prove the claim from the 
introduction that the polynomial 
approximants we construct in the proof of Theorem 
\ref{weier++}  have derivatives that 
diverge pointwise almost everywhere. 

Recall from Section \ref{extreme sec} that 
the derivatives of each of  our approximants are of 
the form $T_n(x,y)$, i.e., 
a perturbed version of the  Chebyshev polynomial $T_n$.
We want to show that $\frac {d}{dx} T_n(x,y)$ 
is large whenever $T_n(x,y)$ is small, so that the set
were $T_n(x,y)$ is close to zero has small measure.

\begin{lemma} \label{small set}
Suppose $T_n(x,y)$ is a perturbed Chebyshev function as in Lemma 
\ref{sup bound}.  If $\lVert y\rVert < t$ and $t$ is small enough, then 
$|\{x \in [-1,1]: |T_n(x,y)| < t \}| =O(t)$. 
\end{lemma} 

\begin{proof}
Suppose $r^n_k(y)$ is the $k$th  root of the perturbed polynomial  $T_n(x,y)$.
Recall that this is a perturbation of $r^n_k$, the $k$th root of $T_n(x)$ 
(possibly $r^n_k$ is equal to $r^n_k(y)$).  The point $r^n_k$ is the left 
endpoint of the $k$th nodal interval  $I^n_k$ of $T_n(x)$.

Write $T_n(x,y) = (x-r^n_k) S^n_k(x)$, i.e., 
$S^n_k$ is a constant $A$  times the product of terms $(x-r^n_j)$  over all the 
roots other than $r^n_k$. 
Let $x_k \in I^n_k$ and $x_{k+1} \in I^n_{k+1}$ 
be points where $|T_n(x,y)|$ is maximized in $  I^n_k$ and $I^n_{k+1}$ respectively. 
Then 
\begin{align} \label{f_k bound 1} 
	\quad 
|S^n_k(x_{k})| = \frac {|T_n(x_k, y)|}{|x_{k}-r_k|}  \geq \frac {1/2}{|I_k|}
\end{align} 
and 
\begin{align} \label{f_k bound 2} 
|S^n_k(x_{k+1})|=\frac {|T_n(x_{k+1}, y)|}{|x_{k+1}-r_{k}|} 
\geq \frac {1/2}{|I_{k+1}|},
\end{align}
since by Lemma \ref{sup bound} the extreme values of $|T|$ are
bigger than $1/2$ if $t$ is small enough.
In both cases, the lower bound is bigger than $1/(2|J_k|)$ where 
 $J_k = I^n_k \cup I^n_{k+1} $. 
By Lemma \ref{length growth},  adjacent nodal intervals have comparable length, 
so $ |J_k| \simeq |I^n_k|$. 

Since 
\[ \log |S^n_k(x)| = \log |A| +\sum_{j:j \ne k} \log |x-r_j|,\]
and each individual term is concave down on $ [x_k, x_{k+1}]$, we see 
that $\log |S^n_k|$ is concave down here as well. Hence $\log |S^n_k|$, and 
therefore $|S^n_k|$,  attains its minimum over $[x_k, x_{k+1}]$
at  one of the endpoints, i.e.,  at either $x_k$ or $x_{k+1}$.
By Equations (\ref{f_k bound 1}), (\ref{f_k bound 2}),
and the remarks  following them, 
we deduce that $|S^n_k(x)| \geq C/ |J_k|$ for all $  x \in [x_k , x_{k+1}]$.
Thus 
\begin{align*}
\{ x \in J_k : |T_n(x,y)| <  \epsilon \} 
	&=
\{ x \in J_k: |x-r_k|\cdot |S^n_k(x)|  < \epsilon \}  \\
	&\subset
\{ x \in J_k:  \frac{C |x-r_k| }{|J_k|}   < \epsilon \} ,
\end{align*}
and hence 
\begin{align*}
|\{ x \in J_k: |T_n(x,y)|  < \epsilon \}| 
\leq 
|\{ x \in J_k: |x-r_k| <   \epsilon |J_k|/C \}|
 = 2 \epsilon |J_k|/C.
\end{align*}
Note that the intervals $J_k$  cover each $I^n_k$ twice,
so $\sum |J_k| \leq 2 \sum_k |I^n_k| \leq  4$.
Therefore, 
\begin{align*}
|\{ x \in [-1,1]: |T_n(x,y)|  < \epsilon \}| 
\leq 
\frac {2 \epsilon}{C} \sum_k |J_k| \leq   \frac{8 \epsilon }{C} = O(\epsilon). \qedhere
\end{align*}
\end{proof} 

\begin{cor} \label{diverges}
If $f $ in Theorem \ref{weier++} is Lipschitz, the sequence $\{p_n\}$ can 
be chosen so that $\{p_n'\}$ is uniformly bounded and so that  any 
sub-sequence diverges pointwise  Lebesgue almost everywhere on $[-1,1]$.
\end{cor} 

\begin{proof} 
As in the proof of Theorem \ref{weier++}, 
it suffices to consider $f$ to be $t$-Lipschitz for 
some fixed $t>0$, and we showed in  
Section \ref{fixed pt sec} that such  a function
can be approximated by polynomials whose derivatives are
functions of the form $T_n(x,y_n)$ with $|y_n| \leq t$.
From the proof of Lemma \ref{small set}, we can deduce that 
the  sets 
	\[ N_n =\{ x \in [-1,1]: T_n(x,y_n)  <- 1/2 \},\]
	\[P_n=\{ x \in [-1,1]: T_n(x,y_n)  > 1/2 \}\]
each have density bounded uniformly away from zero 
in any interval $[a,b] \subset [-1,1]$ if $n$ is 
large enough, i.e., 
\[ \liminf_{n \to \infty} \frac {|P_n \cap [a,b]|}{b-a}   \geq c > 0.\]
The same inequality holds for $N_n$. 
The Lebesgue differentiation theorem (e.g., Theorem 3.21 in 
\cite{MR1681462}) then implies that for  every 
subset $E \subset [-1,1]$ of positive measure, there is an 
interval $J \subset [-1,1]$ so that $|E\cap J| > (1-c/2)|J|$, 
e.g., take a small enough interval around a point of 
density of $E$. Therefore $|E\cap P_n\cap J|$ and 
$|E\cap N_n \cap J|$ both have measure 
larger than $(c/2)|J| >0$ for all large enough $n$ 
(depending on $E$). However, if $E^+$ 
is the set  of $x$'s where $\{T_{n_k}(x, y_{n_k})\}$  
converges to a non-negative limit for some sequence $\{n_k\}$, 
then $E^+$ is disjoint from $N_n$ for all sufficiently large $n$, 
hence $E^+$ must have Lebesgue measure zero. 
Similarly,  if  $E^-$ is the set where  the sequence 
converges to a non-positive limit, then $E^-$  is disjoint from 
$P_n$ for all large enough $n$, and hence $E^-$ also has zero measure. 
Thus the sequence of derivatives  diverges almost everywhere. 
\end{proof} 

The Clunie-Kuijlaars theorem discussed in the introduction implies that 
if the sequence  $\{p_n'\}$ does not converge uniformly on
$[-1,1]$ to a Laguerre-P{\'o}lya
function, then at almost every point of $[-1,1]$ it either  
diverges, converges to $0$, or $|p_n'|$ converges to $\infty$. 
Corollary \ref{diverges}  shows the first option can occur, 
and \cite{weier_part2} shows  
that approximants can  be chosen so that $p_n'(x)$ converges 
to zero almost everywhere, or so that it tends to either $+\infty $ 
or $-\infty$ almost everywhere. 

\section{Theorem \ref{weier++} fails for some subsets of $\reals$} \label{fails sec}

For $X \subset \reals$, let
${{CP}}(X)$ denote all uniform limits of real polynomials
on $X$ with all critical points contained in $X$.
Theorem \ref{weier++} says that $CP(X) = C_\reals(X)$ when 
$X$ is an interval, but we will show that 
this fails for some disconnected  subsets  of $\reals$.
If $X$ has only one or two points, the situation is trivial  
(non-constant linear functions have no critical points and can 
approximate every function on $X$),
but even for three points the answer is not immediately obvious.
If $ X=\{ -1,0,1\}$, then any polynomial $p$ with
critical points in $X$ has a derivative of the form
\[ C \int (x+1)^ax^b(x-1)^c dx,\]
for some triple $(a,b,c)$ of non-negative integers.  Thus
\[ \frac{p(1)-p(0) }{p(0)-p(-1)} =
\frac { \int_0^1(x+1)^ax^b(x-1)^c dx}
{ \int_{-1}^0(x+1)^ax^b(x-1)^c dx},\]
and this  takes only countably many different values.
Thus $p$ restricted to $X$  cannot equal every possible real-valued
function on these three points.
However, with some work, one can show that the set of possible ratios
is dense in $\reals$, and this implies $CP(X)  = C_\reals(X)$.
Very briefly, consider  $ p(x)= (1-x)^a(1+x)^{n-a}$ with $n$ large
and $0\leq a\leq n$.
Normalize $p$  so the
maximum of $p$  in $[-1, 1]$ is $1$. This function has  a single bump
with  max that slides from $-1$ to $1$  as $a$ goes from $0$ to $n$.
The width  of the bump is about $1/\sqrt{n}$, but the distance
between the peaks for consecutive $a$'s
is about $1/n$. Thus only about $1/\sqrt{n}$ of the mass moves across
$0$ as we increment $a $ near $n/2$, and careful estimates show that
we can approximate any positive ratio that we want.
However, for four points, the situation is different.

\begin{figure}[htb]
\centerline{
\includegraphics[height=2.0in]{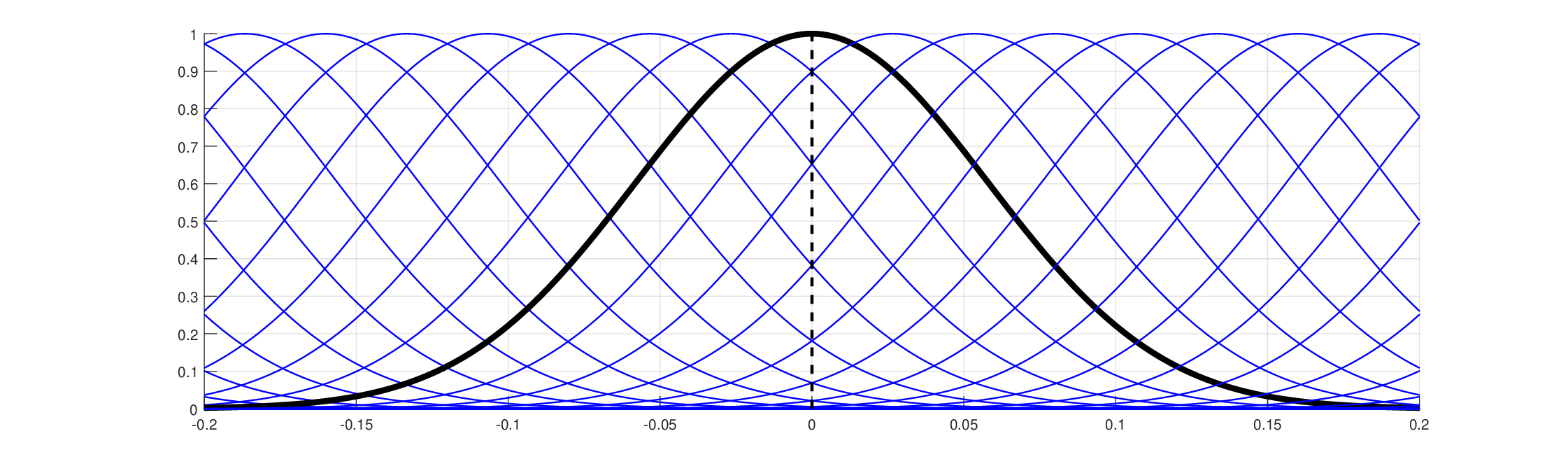}
}
\caption{ \label{three_pts-slide}
Plots of $ C \cdot \int_0^t (x+1)^a(x-1)^{n-a} dx$, where
$n=300$, $a=2,4,\dots, 298$ and $C=C(a,n)$ is chosen so the 
function's maximum is $1$.
The thicker graph corresponds to $a = n/2$.
By taking $n$ large and ``sliding'' $a$ from $1$ to $n$, the
ratio of the areas over $[-1,0]$ and $[0,1]$ can approximate any positive value
we wish.
}
\end{figure}

\begin{lemma}\label{four points}
Suppose that  $0<\epsilon <1/4$.
Suppose $X$ is a compact set
contained in $[ -1-\epsilon, -1] \cup[ 1, 1+\epsilon]$, and  that
it contains at least two distinct points in each of these intervals.
Then there is a real-valued function continuous  on $X$ that cannot be
uniformly approximated by polynomials  that have
all their critical points  inside $X$.
\end{lemma}

\begin{proof}
Let $X_{-1} = X \cap [-1-\epsilon, 1]$ and
$X_1 = X \cap [1, 1+\epsilon]$. Let $s = \inf X_{-1}$,
$t = \sup X_{-1}$, $u = \inf X_1$, and $v= \sup X_1$.
We claim there is a $\lambda <1$
depending only on $\epsilon$ so that 
\begin{align} \label{either or} 
	\min(|p(s)-p(t)|, |p(u)-p(v)|)  \leq \lambda |p(t)-p(u)|
\end{align} 
holds for every real polynomial with all critical points in $X$.
If such a $\lambda $ exists, then a  function $f$  on $X$ such that
\[  |f(s)-f(t)|  =|f(t)-f(u)| =|f(u)-f(v)| >0 \]
cannot be uniformly approximated by polynomials in $CP(X)$.

Take a $p$ as above and let $p' $ be its derivative. Rescaling
by a constant we may assume that $p'$ is monic
and that it has $n$ zeros in $X_{-1}$ and $m$ zeros in $X_1$.
Note that $p'$ has  a single sign in $[t,u]$; without loss of
generality we may assume $p'>0$   here.
For $x \in [0, \epsilon]$, the distance of $x$ to
$X_{-1}$ is $\geq 1$ and its distance to $X_1$
is  $\geq 1-\epsilon$. Thus  on
$[0, \epsilon]$ we have $|p'| \geq (1-\epsilon)^m$.
Similarly, $|p'| \geq (1-\epsilon)^{n}$ on $[-\epsilon,0]$.
Thus
\[|p(u) -p(t)| \geq \Big|\int_t^u p' \Big|
	\geq  \max \bigg(\Big|\int_{-\epsilon}^0 p' \Big|, \Big|\int_0^{\epsilon} p' \Big| \bigg)
\geq  \max\big(\epsilon (1-\epsilon)^{n}, \epsilon(1-\epsilon)^m\big).\]
Every point of $[1, 1+\epsilon]$ is within distance $\epsilon$ of
the $m$ roots contained in $X_1$ and is within distance $2 + 2 \epsilon$
 the $n$ roots in $X_{-1}$.  Therefore on
$X_1$ we have $|p'| \leq \epsilon^m (2+2 \epsilon)^n$.
Similarly, on $X_{-1}$ we have $|p'| \leq \epsilon^n (2+ 2 \epsilon)^m$.
Thus
\[|p(s) -p(t)| \leq |\int_s^t p' | \leq   \epsilon^{n+1}(2+2\epsilon)^m,\]
\[|p(u) -p(v)| \leq |\int_u^v p' | \leq   \epsilon^{m+1}(2+2\epsilon)^n.\]
The inequality $\min(x,y) \leq \lambda \max(w,z) $ follows from $xy \leq  \lambda^2 wz$,
so to prove (\ref{either or})  it suffices to verify that
\[
\big( \epsilon^{n+1}(2+2\epsilon)^m \big)
\big( \epsilon^{m+1}(2+2\epsilon)^n \big)
\leq \lambda^2 \epsilon^2 (1-\epsilon)^{n+m}
\]
\[
 \bigg(  \frac  {\epsilon(2+2\epsilon)}
          {1-\epsilon}\bigg)^{n+m}
 \leq \lambda^2 .
\]
This holds  for some $\lambda<1$ if
$
  {\epsilon(2+2\epsilon)} <    1-\epsilon.
  $
Using  the quadratic formula,  we get 
\[\frac{-3 -\sqrt{17}}{4}  < \epsilon <\frac{-3+ \sqrt{17}}{ 4} \approx .2808 .\]
The right side is larger than $1/4$,  so
this proves the lemma.
\end{proof}

\begin{ques}
For which compact sets $X \subset \reals$ is $CP(X) = C_\reals(X)$? 
Does this fail for all disconnected sets  $X$ with more than three points?
\end{ques}

\bibliographystyle{rminumeric}
\bibliography{weier}

\end{document}